\documentclass[12pt, reqno]{amsart}
\usepackage{amssymb, mathrsfs}
\usepackage{latexsym}
\usepackage{amsmath}
\usepackage{amsthm}
\usepackage{amsfonts}
\usepackage{dsfont, upgreek}
\usepackage{mathtools}
\usepackage{epsfig}
\usepackage{amscd}
\usepackage{graphicx}
\usepackage{tikz-cd }

\usepackage{enumitem}
\setlist[enumerate]{itemsep=0.5ex}

\usepackage{color}
\usepackage{tikz}
\usetikzlibrary{patterns}

\usepackage[left=1.4in,right=1.4in,top=1.2in,bottom=1.2in]{geometry}

\usetikzlibrary{decorations.markings}
\usetikzlibrary{arrows.meta}

\usepackage{extarrows}

\usepackage{adjustbox}
\usepackage{pgfplots}
\usepgfplotslibrary{colormaps}
\usepackage{slashed}

\usepackage[colorlinks=true,linkcolor=blue,citecolor=blue]{hyperref}

\usepackage[colorinlistoftodos,prependcaption,textsize=tiny]{todonotes}  

\usepackage[all,cmtip]{xy}
\theoremstyle{plain}
\newtheorem{theorem}{Theorem}[section]
\newtheorem{proposition}[theorem]{Proposition}
\newtheorem{lemma}[theorem]{Lemma}

\newtheorem{question}[theorem]{Question}

\theoremstyle{definition} 
\newtheorem{definition}[theorem]{Definition}

\newtheorem{example}[theorem]{Example}
\newtheorem*{claim*}{Claim}
\newtheorem{claim}[theorem]{Claim}

\theoremstyle{remark} 
\newtheorem{remark}[theorem]{Remark}

\numberwithin{equation}{section}

\newcommand{\Sc}{\mathrm{Sc}}

\newcommand{\R}{\mathbb{R}}
\newcommand{\dist}{\mathrm{dist}}
\newcommand{\Endo}{\mathrm{End}}
\newcommand{\Bigwedge}{\mathord{\adjustbox{raise=.4ex, totalheight=.7\baselineskip}{$\bigwedge$}}}
\newcommand{\Swedge}{\mathord{\adjustbox{raise=.4ex, totalheight=.55\baselineskip}{$\bigwedge$}}}

\newcommand{\ind}{\textup{Ind}}
\newcommand{\id}{\mathrm{id}}

\newcommand{\sph}{\mathbb{S}}

\newcommand{\Ric}{\mathrm{Ric}}
\newcommand{\tcon}{\widehat \nabla}
\newcommand{\tdirac}{\widehat D}
\newcommand{\grad}{\text{grad}}

\newcommand{\expand}{\rho}

\newcommand{\interior}[1]{%
	{\kern0pt#1}^{\mathrm{\,o}}%
}

\makeatletter
\let\save@mathaccent\mathaccent
\newcommand*\if@single[3]{%
	\setbox0\hbox{${\mathaccent"0362{#1}}^H$}%
	\setbox2\hbox{${\mathaccent"0362{\kern0pt#1}}^H$}%
	\ifdim\ht0=\ht2 #3\else #2\fi
}
\newcommand*\rel@kern[1]{\kern#1\dimexpr\macc@kerna}
\newcommand*\wideaccent[2]{\@ifnextchar^{{\wide@accent{#1}{#2}{0}}}{\wide@accent{#1}{#2}{1}}}
\newcommand*\wide@accent[3]{\if@single{#2}{\wide@accent@{#1}{#2}{#3}{1}}{\wide@accent@{#1}{#2}{#3}{2}}}
\newcommand*\wide@accent@[4]{%
	\begingroup
	\def\mathaccent##1##2{%
		\let\mathaccent\save@mathaccent
		\if#42 \let\macc@nucleus\first@char \fi
		\setbox\z@\hbox{$\macc@style{\macc@nucleus}_{}$}%
		\setbox\tw@\hbox{$\macc@style{\macc@nucleus}{}_{}$}%
		\dimen@\wd\tw@
		\advance\dimen@-\wd\z@
		\divide\dimen@ 3
		\@tempdima\wd\tw@
		\advance\@tempdima-\scriptspace
		\divide\@tempdima 10
		\advance\dimen@-\@tempdima
		\ifdim\dimen@>\z@ \dimen@0pt\fi
		\rel@kern{0.6}\kern-\dimen@
		\if#41
		#1{\rel@kern{-0.6}\kern\dimen@\macc@nucleus\rel@kern{0.4}\kern\dimen@}%
		\advance\dimen@0.4\dimexpr\macc@kerna
		\let\final@kern#3%
		\ifdim\dimen@<\z@ \let\final@kern1\fi
		\if\final@kern1 \kern-\dimen@\fi
		\else
		#1{\rel@kern{-0.6}\kern\dimen@#2}%
		\fi
	}%
	\macc@depth\@ne
	\let\math@bgroup\@empty \let\math@egroup\macc@set@skewchar
	\mathsurround\z@ \frozen@everymath{\mathgroup\macc@group\relax}%
	\macc@set@skewchar\relax
	\let\mathaccentV\macc@nested@a
	\if#41
	\macc@nested@a\relax111{#2}%
	\else
	\def\gobble@till@marker##1\endmarker{}%
	\futurelet\first@char\gobble@till@marker#2\endmarker
	\ifcat\noexpand\first@char A\else
	\def\first@char{}%
	\fi
	\macc@nested@a\relax111{\first@char}%
	\fi
	\endgroup
}
\makeatother

\newcommand\overbar{\wideaccent\overline}

\makeatletter
\newcommand*{\transpose}{%
	{\mathpalette\@transpose{}}%
}
\newcommand*{\@transpose}[2]{%
	\raisebox{\depth}{$\m@th#1\intercal$}%
}
\makeatother

\begin{document}
	
\title{Scalar curvature rigidity of degenerate warped product spaces }

\author{Jinmin Wang}
\address[Jinmin Wang]{Department of Mathematics, Texas A\&M University}
\email{jinmin@tamu.edu}
\thanks{The first author is partially supported by  NSF  1952693.}
\author{Zhizhang Xie}
\address[Zhizhang Xie]{ Department of Mathematics, Texas A\&M University }
\email{xie@tamu.edu}
\thanks{The second author is partially supported by NSF  1952693.}

\begin{abstract}
	In this paper we prove the scalar curvature extremality and rigidity for a class of warped product spaces that are possibly degenerate at the two ends. The leaves of these warped product spaces can be any closed Riemannian manifolds with nonnegative curvature operators and nonvanishing Euler characteristics, flat tori, round spheres and their direct products. In particular, we obtain the scalar curvature extremality and rigidity for certain degenerate toric bands and also for round spheres with two antipodal points removed. This answers positively the corresponding questions of Gromov in all dimensions. 
\end{abstract}
\maketitle

\section{Introduction}
Scalar curvature extremality and rigidity problems occupy a central role in Riemannian geometry. The first examples of  scalar curvature  rigidity are flat tori and standard round spheres. Specifically, Schoen--Yau \cite{MR535700, MR541332} and Gromov--Lawson \cite{MR569070} showed that the torus $\mathbb T^n$ admits no metric of positive scalar curvature, and any metric on $\mathbb T^n$ with nonnegative scalar curvature is a flat metric.  For the standard round sphere $(\sph^n, g_{st})$, Llarull proved that if  $g$ is a Riemannian metric on   $\sph^n$ such that $g\geq g_{st}$ and $\Sc_g \geq \Sc_{g_{st}}$, then $g= g_{st}$  \cite{Llarull}. Here $\Sc_g$ stands for the scalar curvature of $g$. 

Goette and Semmelmann generalized  the theorem of Llarull and proved the scalar curvature extremality and rigidity for all closed manifolds with non-vanishing Euler characteristics that are equipped with metrics having nonnegative curvature operators \cite{GoetteSemmelmann}. Later on, Lott extended their theorem  to a scalar-and-mean curvature extremality and rigidity theorem for compact manifolds with smooth boundary \cite{Lottboundary}. 

 Inspired by Gromov's $\mu$-bubble approach to scalar curvature problems, Cecchini and Zeidler proved a scalar-and-mean curvature extremality and rigidity \cite{Cecchini:2021vs} for the following class of compact warped product spaces: $(X\times I, g)$ satisfying that  $X$ is closed spin manifold with non-zero Euler characteristic, $I = [a, b]$ is a closed finite interval, and $g$ is a warped product metric of the form: 
 \[  g = dr^2 + \varphi(r)^2 g_X \]
 such that $g_X$ is a metric on $X$ with nonnegative curvature operator and $\varphi$ is a strictly log-concave positive function on $I = [a, b]$.  In contrast with the results of Llarull \cite{Llarull}, Goette--Semmelmann \cite{GoetteSemmelmann} and Lott \cite{Lottboundary}, the result of Cecchini--Zeidler only requires the metric of the leaf $X$ to have nonnegative curvature operator, rather than requiring the entire underlying manifold to satisfy this condition. For example,  it is applicable to  annuli in odd dimensional hyperbolic spaces, where an annulus is viewed as a warped product space of the form $\sph^{n-1}\times [a, b]$.  More recently, the authors have generalized the theorem of Cecchini--Zeidler  and obtained a dihedral ridigity theorem for a class of codimension zero compact \emph{submanifolds} with polyhedral corners in warped product spaces \cite{Wang:2023warp}. It is worth noting that these submanifolds themselves are not necessarily warped product spaces and may have faces that are neither orthogonal nor parallel to the radial direction of the warped product metric.
 
 Thus far, all the aforementioned results have primarily focused on addressing the scalar curvature extremality and rigidity problem in the context of \emph{compact}, hence \emph{complete}, manifolds. From a technical standpoint, this emphasis on completeness is crucial for making sense of the relevant index theory. Therefore, it is remarkable that Gromov, using his $\mu$-bubble techniques, managed to establish  scalar curvature extremality and rigidity for certain \emph{incomplete}  warped product spaces \cite{Gromov4lectures2019}, which we will refer to as  \emph{degenerate warped product spaces} from now on.  More precisely, Gromov sketched a proof of the scalar curvature rigidity for certain degenerate toric bands \cite{Gromov4lectures2019} in dimensions $n+1\leq 8$. These toric bands $\mathbb T^{n}\times (-\frac{\pi}{n}, \frac{\pi}{n})$ are equipped with the warped product metrics 
\[ g= dr^2 + \varphi(r)^2 g_0\]
where $g_0$ is a flat metric on $\mathbb T^{n}$ and 
\[ \varphi(r) = \left(\cos \frac{n r}{2}\right)^{2/n}. \]
In the same paper \cite{Gromov4lectures2019}, Gromov also sketched a proof for  the scalar rigidity for  the $n$-dimensional standard round sphere with two antipodal punctures, denoted as $(\sph^n\backslash\{\pm\}, g_{st})$,   in dimensions $3\leq n \leq 8$ (cf. \cite{Hirsch:2022aa} and \cite{MR4594778}   for the dimension three case). One key observation made by Gromov is to view the space  $(\sph^n\backslash\{\pm\}, g_{st})$  as a warped product space
\[  g_{st} = dr^2 + \cos(r)^2 g^{\sph^{n-1}}_{st}  \]
  where $r\in (-\pi/2, \pi/2)$ and $g^{\sph^{n-1}}_{st} $ is the standard round metric on $\sph^{n-1}$. Note that the dimensional restriction in both of Gromov's results arises due to the usual regularity issue encountered in minimal hypersurface theory.

In this paper, we generalize the results of Gromov and prove the scalar curvature extremality and rigidity for a fairly large class of degenerate warped product spaces in  all dimensions. The main class of warped product spaces we consider is the following. Let $M=(-c,c)\times X$ be an $n$-dimensional manifold equipped with the following warped product metric 
$$g=dr^2+\varphi(r)^2g_X.$$ 
The leaf $X$ is allowed to be the Riemannian product of finitely many spaces from any of the following classes of closed manifolds: 
\begin{enumerate}[label=(\roman*)]
	\item round spheres of any dimension,
	\item closed Riemannian manifolds with nonnegative curvature operators and nonvanishing Euler characteristics, and 
	\item flat tori.
\end{enumerate}
The warping function $\varphi$ is required to be  \emph{admissible} in the following sense. 

\begin{definition}\label{def:admissible}
 We say a warping function $\varphi$ is \emph{admissible} if $\varphi$ satisfies the following properties:
 \begin{enumerate}
 	\item $\varphi$ is log-concave, that is, $(\log\varphi)''\leq 0$,
 	\item $\varphi(r)>0$ for $r\in(-c,c)$, and $\lim_{r\to \pm c} \varphi(r)=0$,
 	\item there exists a small $\varepsilon>0$ such that $\psi'(r)+n\,\psi(r)^2/2$ on the interval $(c, c-\varepsilon)$ is non-decreasing, and $\psi'(r)+n\,\psi(r)^2/2$ on the interval $(-c,-c+\varepsilon)$ is non-increasing,   where $\psi=(\log\varphi)'$ and $n=\dim M$.
 \end{enumerate}  
\end{definition}
The log-concavity of $\varphi$ is a commonly expected necessary condition for the scalar curvature extremality and rigidity of warped product spaces. However, the above definition introduces a new condition (3), which, to the best of the authors' knowledge, has not been previously considered in the literature regarding scalar curvature extremality and rigidity. Despite its somewhat technical nature, condition (3) is shown to be necessary through Example \ref{ex:nonrigid} below. More precisely, Example \ref{ex:nonrigid} shows that if we drop condition (3), then scalar curvature extremality and rigidity \emph{fail} for certain degenerate toric bands with warping functions satisfying conditions (1) and (2).

Before we state the main theorem of the paper, we recall the definition of spin maps. 
\begin{definition}
	A  map $f\colon N \to M$ between two oriented  manifolds $N$ and $M$ is called a spin map if the second Stiefel--Whitney classes of $TM$ and $TN$ are related by 
	\[  w_2(TN) = f^\ast(w_2(TM)).\]
	Equivalently, $f\colon N \to M$ is a spin map   if $TN\oplus f^\ast TM$ admits a spin structure. 
\end{definition}

We have the following main theorem of the paper.

\begin{theorem}\label{thm:extremality}
	Let $M=(-c,c)\times X$ be an $n$-dimensional manifold equipped with the warped product metric 
	$$g=dr^2+\varphi(r)^2g_X$$
such that 
	\begin{enumerate}[label=$(\arabic*)$]
		\item $\varphi$ is admissible in the sense of Definition $\ref{def:admissible}$ and 
		\item $(X,g_X)$ is the Riemannian product of finitely many spaces from the classes \textup{(i)--(iii)} listed above.
	\end{enumerate}
	Let $(N,\overbar g)$ be a Riemannian manifold and $f\colon N\to M$ be a smooth spin proper map with non-zero degree. If $f$ is distance non-increasing and $\Sc_{\overbar g}\geq f^*\Sc_g$, then $\Sc_{\overbar g}=f^*\Sc_g$. Furthermore, the following hold. 
		\begin{enumerate}[label=$(\mathrm{\Roman*})$]
		\item If $\varphi$ is strictly log-concave, that is, $(\log\varphi)''<0$, then $N=(-c,c)\times Y$ for some Riemannian manifold $(Y,g_Y)$ and the metric  $\overbar g=dr^2+\varphi(r)^2g_Y$, and the map $f$ respects  the product structures. 
		\item If $\varphi$ is strictly log-concave and the metric $g_X$ on the leaf $X$ has positive Ricci curvature, then $f$ is a local isometry.
	\end{enumerate}
\end{theorem}
Our approach uses the index theory of twisted Dirac operators coupled with potentials. However, due to the noncompactness of the underlying space and the incompleteness of the metric, it seems unfeasible to hope for a general index theory on the entire underlying space. To get around this, we focus on codimension zero compact submanifolds with boundary of the underlying space, where the classical index theory for manifolds with boundary can be  applied. However, this approach inevitably introduces additional error terms when comparing various geometric quantities,  such as scalar curvatures and mean curvatures. To overcome this difficulty, a key aspect of our proof involves carefully balancing these extra error terms with the comparison conditions given by our assumptions. We show that the failure of the conclusions of our main theorem would yield a geometric term that, via a Poincar\'e type inequality (Lemma \ref{lemma:poincare}),   ultimately dominates these additional error terms. This leads to a contradiction, hence proves our theorem. 

We remark that  the case where the leaf $(X, g_X)$ is an odd dimensional sphere or torus, the  vanishing of the Euler characteristic of the leaf imposes an extra difficulty. When  the leaf is an odd dimensional sphere,   we follow Llarull's idea of taking the  direct product with a large circle\footnote{More precisely, one also needs to consider the smashed product of a sphere with a circle. Note that the smashed product of an odd dimensional sphere with a circle is an even dimensional sphere, where the latter has nonzero Euler characteristic.} \cite[Section 4]{Llarull}. When the leaf is a torus, we pair the Dirac operator with an almost flat bundle. In both cases,  the corresponding procedure introduces an extra small error term in the relevant curvature estimates.  A key step of our proof is to dominate this extra error term by again  a  Poincar\'e type inequality. 

In fact, due to the extra error term caused by introducing an auxiliary circle, there is a minor gap in Llarull's  proof  for the scalar rigidity of a \emph{closed} standard \emph{odd} dimensional sphere \cite[Section 4]{Llarull}. We make the observation that the minor gap in Llarull's original arugment can be fixed by applying the Poincar\'e type inequality we mentioned above. 
 
\begin{theorem}[{Llarull \cite{Llarull}}]\label{thm:llarull}
Let $\mathbb S^{2k+1}$ be the $(2k+1)$-dimensional standard round sphere. Let $(N,\overbar g)$ be a closed spin Riemannian manifold and  $f\colon N\to M$  a smooth map with non-zero degree. If  $\Sc_{\overbar g}\geq 2k(2k+1)$ and $f$ is area-non-increasing,  then $f$ is an isometry.
\end{theorem}
Of course, if we artificially remove two antipodal points of $\mathbb S^{2k+1}$ and view it as a warped product space, Theorem \ref{thm:llarull} appears to be a special case of  Theorem \ref{thm:extremality}.  However, it is important to note  the different assumptions on the map $f$.  The map $f$ is only assumed to \emph{area-non-increasing} in Theorem \ref{thm:llarull}, as opposed to being distance-non-increasing in Theorem \ref{thm:extremality}. It is worth pointing out that in general we cannot replace the assumption that $f$ is distance-non-increasing in Theorem \ref{thm:extremality} by the weaker assumption that $f$ is area-non-increasing. On the other hand, our proof shows that Theorem \ref{thm:extremality} still holds under the  weaker assumption that $f$ is distance-non-increasing along the warping direction and area-non-increasing along the leaf direction. More precisely, let us write $f(x) = (r, z) \in M = (-c, c)\times X$, and  $X_r = \{r\}\times X$ equipped with metric $\varphi(r)^2 g_X$. Define $P$ to be the orthogonal  projection from $T_{f(x)}M$ to $T_{f(x)} X_r$.  Then instead of being distance-non-increasing, we only need to assume the function $f$ in Theorem \ref{thm:extremality} to satisfy that  $f^\ast r$ over $N$ is $1$-Lipschitz and $Pf_\ast: T_xN \to T_{f(x)}X$ is area-non-increasing for all $x\in N$.  

In the special case where the leaf $X$ of $M = (-c, c)\times X$ is a flat torus, we have the following slight improvement of Theorem \ref{thm:extremality}.   

\begin{theorem}\label{thm:extremalityTorus-intro}	Let $M=(-c,c)\times X$ be an open manifold and
	$$g=dr^2+\varphi(r)^2g_X$$
	a warped product metric on $M$ such that
	\begin{enumerate}
		\item $\varphi$ is admissible in the sense of Definition \ref{def:admissible}, and
		\item $(X,g_X)$ is the torus $\mathbb T^{n-1}$ equipped with a flat metric $g_o$.
	\end{enumerate}
	Let $(N,\overbar g)$ be a spin Riemannian manifold and $f\colon N\to M$ be a smooth proper map with non-zero degree. If the function $f^*r$ over $N$ is of Lipschitz constant at most $1$ and $\Sc_{\overbar g}\geq f^*\Sc_g$, then $\Sc_{\overbar g}=f^*\Sc_g$. Furthermore, if in addition  $\varphi$ is strictly log-concave, that is, $(\log\varphi)''<0$, then $N=(-c,c)\times Y$, the map $f$ respects  the product structures, and the metric  $\overbar g$ is also a warped product metric of the form \[ \overbar g = dr^2+\varphi(r)^2g_Y,  \]
	where $g_Y$ is a flat metric on $Y$. 
\end{theorem}

Recall that the $n$-dimensional standard round sphere with two antipodal punctures $(\sph^n\backslash\{\pm\}, g_{st})$ may be viewed as a warped product space
\[  g_{st} = dr^2 + \cos(r)^2 g^{\sph^{n-1}}_{st}, \]
 where $r\in (-\pi/2, \pi/2)$. It is easy to verify that the function $\varphi(r) = \cos(r)$ is admissible in the sense of Definition \ref{def:admissible}. As a special case of Theorem \ref{thm:extremality},  we have the following theorem, which generalizes the corresponding result of Gromov to all dimensions. 

\begin{theorem}\label{coro:sphere}
	Assume that $n\geq 3$. Let $M$ be the $n$-dimensional standard round sphere with a pair of antipodal points removed. Let $(N,\overbar g)$ be an open spin Riemannian manifold. Let $f\colon N\to M$ be a proper smooth map with non-zero degree. If $f$ is distance-non-increasing and $\Sc_{\overbar g}\geq n(n-1)$, then $f$ is an isometry.
\end{theorem}

We would like to mention that Theorem  \ref{coro:sphere} was also  obtained independently in a  preprint of B\"{a}r-Brendle-Hanke-Wang \cite{Baer:2023aa}.

So far, we have mainly focused on scalar rigidity results on bands that are degenerate at both ends. It is not difficult to see that our techniques can be adapted to prove the following scalar-and-mean curvature rigidity for warped product spaces that are degenerate at one end. 

\begin{theorem}\label{thm:scalar-mean-intro}
	Let $M=[-c,c)\times X$ and
	$$g=dr^2+\varphi(r)^2g_X$$
	a warped product metric on $M$ such that
	\begin{enumerate}[label=$(\arabic*)$]
		\item $\varphi$ is log-concave, that is, $\psi\coloneqq (\log\varphi)''\leq 0$,
		\item $\psi'(r)+n\psi(r)^2/2$ is non-decreasing near $r=c$,
		\item $\varphi(r)>0$ for $r\in [-c,c)$, and $\varphi(c)=0$, and
		\item $(X,g_X)$ is the Riemannian product of finitely many spaces from the classes \textup{(i)--(iii)} listed above.
	\end{enumerate}
	Let $(N,\overbar g)$ be a Riemannian manifold with boundary and $f\colon N\to M$ be a smooth spin proper map with non-zero degree. If $f$ is distance non-increasing, and the scalar curvature and the mean curvature satisfy
	$$\Sc_{\overbar g}\geq f^*\Sc_g,\text{ and }H_{\overbar g}\geq f^*H_g=-(n-1)\psi(-c),$$
	then $\Sc_{\overbar g}=f^*\Sc_g$ and $H_{\overbar g}= f^*H_g=\psi(c)$. Furthermore, the following hold. 
	\begin{enumerate}[label=$(\mathrm{\Roman*})$]
		\item If $\varphi$ is strictly log-concave, then $N=[-c,c)\times Y$ for some Riemannian manifold $(Y,g_Y)$ and the metric  $\overbar g=dr^2+\varphi(r)^2g_Y$, and the map $f$ respects  the product structures. 
		\item If $\varphi$ is strictly log-concave and the metric $g_X$ on the leaf $X$ has positive Ricci curvature, then $f$ is a local isometry.
	\end{enumerate}
\end{theorem}

The following warped metric
$$dr^2+\varphi(r)^2 g^{\sph^{n-1}_{st}},$$
with  $\varphi(r)$ equal to $r$, $\sin(r)$, or $\sinh(r)$,  represents the metric on the geodesic ball in the spaces form Euclidean space, standard round sphere, and hyperbolic space, respectively. It is easy to verify that all three  functions $r$, $\sin(r)$ and  $\sinh(r)$  are admissible in the sense of Definition \ref{def:admissible}.
As an immediate consequence of Theorem \ref{thm:scalar-mean-intro}, we have the following scalar-and-mean curvature  rigidity for geodesic balls in space forms.

\begin{theorem}\label{cor:geoball}
	Let $(M,g)$ be a geodesic ball in a space form.
	Let $(N,\overbar g)$ be a spin Riemannian manifold with boundary and $f\colon N\to M$ a smooth map such that 
	\begin{enumerate}[label=$(\arabic*)$]
		\item 
		$\Sc(\overbar{g})_x \geq \Sc(g)_{f(x)}$ for all $x\in N$, 
		\item $H_{\overbar{g}}(\partial N)_y \geq  H_{g}(\partial M)_{f(y)}$ for all $y\in \partial N$, 
		\item $f$ is distance-non-increasing on $N$, 
		\item the degree of $f$ is nonzero,	
	\end{enumerate} 
	then $f$ is an isometry. 
\end{theorem}
  The authors have previously proved the above theorem for geodesic balls in Euclidean space using a different method \cite[Theorem 1.7]{Wang:2022vf}. Interestingly, the approach presented in \cite{Wang:2022vf} shows that the above theorem is valid not only for geodesic balls but also for all strictly convex domains with smooth boundary in Euclidean space.  This raises a natural question: does the above scalar-and-mean curvature  rigidity theorem extend to strictly convex domains with smooth boundary in hyperbolic space?

This paper is organized as follows. In section \ref{sec:extrem}, we present some key estimates for Theorem \ref{thm:extremality}, with a specific focus on the case where the leaf $X$ has non-zero Euler characteristic. In  Section \ref{sec:sphere}, we prove the special case of Theorem \ref{thm:extremality} where the leaf $X$ is a standard round sphere. Consequently, we obtain the scalar curvature extremality and rigidity for standard round spheres with two antipodal punctures.   In Section \ref{sec:torus}, we prove the scalar curvature extremality and  rigidity for a class of degenerate toric bands.  The general case of Theorem \ref{thm:extremality} then easily follows from the proofs of the three special cases given in Sections \ref{sec:extrem}, \ref{sec:sphere} and \ref{sec:torus}. Additionally,  we give examples of degenerate toric bands to illustrate the necessity of condition (3) in Definition \ref{def:admissible}. Finally, in Section \ref{sec:scalar-mean}, we  prove the scalar-and-mean curvature rigidity for warped product spaces that are degenerate at one end. 

The authors would like to thank the anonymous referees for  helpful comments and particularly for pointing out some technical issues in an earlier version of the paper.

\section{Some estimates and a special case of Theorem \ref{thm:extremality} }\label{sec:extrem}

In this section, we prove some estimates that will be needed in the proof of Theorem \ref{thm:extremality}. In order to make our proof more transparent and also to highlight the subtleties of different cases, we first demonstrate how these estimates are applied in the special  case of Theorem \ref{thm:extremality} where the leaf  $X$ of $M$ is assumed to have non-zero Euler characteristic.  The general case of Theorem \ref{thm:extremality} requires some extra care. More precisely, we shall deal with the case where the leaf $X$ is an odd dimensional round sphere in Section \ref{sec:sphere}, and the case where the leaf $X$ is a flat torus in Section \ref{sec:torus}. Finally, the general case of Theorem \ref{thm:extremality} will be proved by a combination of the above three cases.

\subsection{Some estimates}\label{sec:estimates}
In this subsection, as a preparation, we first prove a series of estimates that will be needed later. Let us fix some notation. Let $\varphi$ be a log-concave positive function on $(-c,c)$ and $\psi=\varphi'/\varphi$.   We fix a closed sub-interval $I_0 = [-a,a]$ in $(-c,c)$ such that
\begin{itemize}
	\item $\varphi$ attains its maximum in the interior of $I_0$,
	\item $(\psi'+n\psi^2/2)'\geq 0$ on $(a, c)$,   and
	\item $(\psi'+n\psi^2/2)'\leq 0$ on $(-c, -a)$. 
\end{itemize}

We have  the following proposition.
\begin{proposition}\label{prop:special}
	Let $M=(-c,c)\times X$ be an $n$-dimensional manifold and
	$$g=dr^2+\varphi(r)^2g_X$$
	a warped product metric on $M$ such that
	\begin{enumerate}[label=$(\arabic*)$]
		\item $\varphi$ is admissible in the sense of Definition $\ref{def:admissible}$ and 
		\item $(X,g_X)$ is a closed Riemannian manifold with nonnegative curvature operator and non-zero Euler characteristic.
	\end{enumerate}
	Let $N$ be a (possibly incomplete) Riemannian manifold and $f\colon N\to M$  a smooth spin proper map with non-zero degree. Then there is no metric $\overbar g$ on $N$ such that
	\begin{itemize}
		\item $f\colon (N,\overbar g)\to (M,g)$ is distance non-increasing, 
		\item $\Sc_{\overbar g}\geq f^*\Sc_g$, and
		\item $\Sc_{\overbar g}>f^*\Sc_g+\varepsilon_0'$ on the preimage of the $\varepsilon_0$-neighborhood of $I_0\times X$ for some $\varepsilon_0>0$ and $\varepsilon_0'>0$.
	\end{itemize}
\end{proposition}


For each $0< \lambda < c$, we denote by $(M_\lambda,g) = ( [-\lambda, \lambda] \times X, g)$ in $M$. In general, $f^{-1}(X_\lambda)$ may not be a submanifold of $N$. But by Sard's theorem and the transversality theory, there exists  a sequence of positive numbers $\{\lambda_i\}$ with  $0<\lambda_i<c$ and $\lambda_i  \to c$ as $i\to \infty$ such that $f^{-1}(X_{\lambda_i})$ is a submanifold of $N$. Similarly,  there exists  a sequence of positive numbers $\{\lambda'_i\}$ with  $0<\lambda'_i<c$ and $-\lambda_i  \to -c$ as $i\to \infty$ such that $f^{-1}(X_{-\lambda'_i})$ is a submanifold of $N$. 

Precisely speaking,  we should work with the submanifolds $[-\lambda'_i, \lambda_i] \times X$ of $M$. But in order to avoid overload of notation, let us  assume without loss of generality that $\lambda_i = \lambda'_i$. Now let us choose $\lambda$ to be one of the $\lambda_i$'s.  In particular,   $N_\lambda=f^{-1}(M_\lambda)$ is a smooth manifold with boundary, and the map $f\colon N\to M$ restricts to a smooth spin map $f\colon N_\lambda\to M_\lambda$ that maps boundary to boundary. It is clear that the degree of $f\colon N_\lambda\to M_\lambda$ equals to the degree of $f\colon N\to M$.

For any $\varepsilon',\varepsilon>0$, there exist $0< \gamma<c$ and a smooth function
\begin{equation}\label{eq:expand}
	\expand\colon [-\gamma,\gamma]\to [-c,c]
\end{equation}
such that
\begin{itemize}
	\item $\expand(\pm\gamma)=\pm c$,
	\item $1\leq \expand'(r)\leq 1+\varepsilon'$ for $r\in \mathcal N_{\varepsilon}(I_0)$, where $\mathcal N_{\varepsilon}(I_0)$ is  the $\varepsilon$-neighborhood of $I_0$, and
	\item $\expand'(r)=1$ for $r\in [-\gamma,\gamma]\backslash \mathcal N_{\varepsilon}(I_0)$.
\end{itemize}

By construction, if we fix $\varepsilon$ and $\varepsilon'$, then $|\expand(r)-r|$ is a positive constant for all $r$ sufficiently close to $\pm\gamma$. We denote this positive constant by  $\kappa(\varepsilon,\varepsilon')$.

For any $\lambda\in[0,\gamma]$ and $\mu=\expand(\lambda)$, we define
\begin{equation}\label{eq:hchi}
		h_\expand\colon (M_\lambda,g)\to (M_\mu,g),~(r,x)\mapsto (\expand(r),x)
\end{equation}
for $r\in[-\lambda,\lambda]$ and $x\in X$. Note that $\|dh_\expand\|\leq 1+\varepsilon$ and $h_\expand$ maps  the leaf $X_r$ to the leaf  $X_{\expand(r)}$.

We shall Proposition \ref{prop:special} by contradiction. Suppose a metric $\overbar g$ on $N$ as described in Proposition \ref{prop:special} exists. Let us denote  
$$h\coloneqq h_\expand\circ f\colon (N_\lambda,\overbar g)\to (M_\mu,g),$$
where the constants  $\varepsilon,\varepsilon',\lambda,\mu$ appearing in the construction of the function  $\expand$ will be specified later. Set $E=S(TN_\lambda\oplus h^*TM_\mu)$ to be the spinor bundle of $TN_\lambda\oplus h^*TM_\mu$ over $N_\lambda$, which exists since $f$ is assumed to be a spin map. The Clifford actions of $TN_\lambda$ and $h^*TM_\mu$ on $E$ are denoted by $\overbar c$ and $c$, respectively. Let $\mathscr E$ be the grading operator on $E$.

Let $\partial_r$ be the unit vector in $h^*TM_\mu$ along the $r$ direction. Let $\nabla$ be the spinorial connection on $E$ naturally induced by  the Levi--Civita connection on $N$ and the pull-back of the Levi--Civita connection on $M$. We define a new connection on $E$  by
\begin{equation}\label{eq:tcon}
	\widehat\nabla_\xi\coloneqq \nabla_\xi+\frac 1 2 c(\nabla^g_{h_*\xi}\partial_r)c(\partial_r),
\end{equation}
where $\nabla^g$ is the Levi--Civita connection of $(M,g)$. A straightforward computation shows that $c(\partial_r)$ is parallel with respect to $\widehat\nabla$, that is, $\tcon c(\partial_r)=0$. 

Let $\widehat D$ be the Dirac operator on $E$ with respect to $\widehat\nabla$,
$$\widehat D=\sum_{i=1}^n \overbar c(\overbar e_i)\widehat \nabla_{\overbar e_i}$$
where $\{\overbar e_i\}_{1\leq i\leq n}$ is  local orthonormal basis of $TN_\lambda$.

Recall that we have 
\begin{equation}
	\psi=\frac{\varphi'}{\varphi}=(\log\varphi)'.
\end{equation}
From now on, we denote by $r\colon M = (-c, c) \times X \to (-c, c)$ the projection to the first component, that is, $r$ maps the leaf $X_t$ to $t$.
We set 
\begin{equation}
	\Psi\coloneqq \frac{n}{2}\cdot  \psi(h^\ast r) \cdot \mathscr E \cdot c(\partial_r), 
\end{equation}
where $h^\ast r$ is the function $r\circ h \colon N_\lambda \to [-\mu, \mu]$,   
and define 
\begin{equation}
	\widehat D_\Psi=\widehat D+\Psi.
\end{equation}
In the following, we shall consider the Fredholm index of $\widehat D_\Psi$ subject to an appropriate local boundary condition. 
\begin{definition}\label{def:boundaryCondition}
		A section $\sigma$ of $E$ over $N_\lambda$ is said to satisfy the local boundary condition $B$  if 
		$$\mathscr E\overbar c(\overbar \nu) c(\mp\partial_r)\sigma=-\sigma,$$
		on $\partial N_\lambda$, where $\overbar \nu$ is the unit inner normal vectors of $\partial N_\lambda$, and $-\partial_r$ (reps. $\partial_r$) is the unit inner normal vector field of $X_{\mu}$ (resp. $X_{-\mu}$).
\end{definition}

First we prove some key estimates. Let  $\mathcal P\colon C^\infty(N_\lambda,E)\to C^\infty(N_\lambda,T^*N_\lambda\otimes E)$ be the Penrose operator defined by
\begin{equation}\label{eq:penrosedef}
	\mathcal P_{\xi}\sigma\coloneqq \widehat\nabla_{\xi}\sigma+\frac{1}{n}\overbar c(\xi)\widehat D\sigma
\end{equation}
for all $\xi\in TN_\lambda$ and all $\sigma\in C^\infty(N_\lambda,E)$.
We have the following identity (cf. \cite[Section 5.2]{spinorialapproach}):
\begin{equation}\label{eq:penrose}
	|\widehat \nabla\sigma|^2=|\mathcal P\sigma|^2+\frac 1 n|\widehat D \sigma|^2
\end{equation}
all $\sigma\in C^\infty(N_\lambda,E)$.

Let $\sigma\in C^\infty(N_\lambda,E)$ be a smooth section of $E$ satisfying the boundary condition $B$ as given in Definition \ref{def:boundaryCondition}.
By the definition of $\widehat D_\Psi$, we have the pointwise equality
\begin{equation}\label{eq:Bsquare}
	\langle \tdirac_\Psi\sigma, \tdirac_\Psi \sigma\rangle
	=|\tdirac\sigma|^2+\langle \Psi\sigma,\tdirac\sigma\rangle+\langle\tdirac\sigma, \Psi\sigma\rangle+\big(\frac{n}{2}\psi(h^*r)\big)^2 |\sigma|^2.
\end{equation}
 over $N_\lambda$. 
By the Stokes formula, we have
\begin{equation}\label{eq:StokesDsquare}
	\int_{N_\lambda} |\tdirac\sigma|^2=\int_{N_\lambda} \langle\tdirac^2\sigma,\sigma\rangle+\int_{\partial N_\lambda} \langle\overbar c(\overbar \nu) \tdirac\sigma,\sigma\rangle
\end{equation}
where $\overbar \nu$ is the inner unit normal vector of $\partial N_\lambda$.
Note that 
\begin{equation}\label{eq:Lichne}
	\tdirac^2=\tcon^*\tcon+\mathcal R, 
\end{equation}
where $\mathcal R$ is the curvature endomorphism of $E$ with respect to $\tcon$. We shall estimate $\mathcal R$ in Lemma \ref{lemma:curvature>=}. Before that, let us observe that 
\begin{equation}\label{eq:Stokesnabla}
	\int_{N_\lambda}\langle \tcon^*\tcon\sigma,\sigma\rangle=\int_{N_\lambda}| \tcon \sigma|^2 +\int_{\partial N_\lambda}\langle\tcon_{\overbar \nu}\sigma,\sigma\rangle
\end{equation} 
again by  the Stokes formula. 

By combining line \eqref{eq:StokesDsquare}, \eqref{eq:Lichne}, \eqref{eq:Stokesnabla} and \eqref{eq:penrose}, we obtain that
\begin{equation}\label{eq:D^2}
	\begin{split}
		\int_{N_\lambda} |\tdirac\sigma|^2=&\frac{n}{n-1}\int_{N_\lambda} |\mathcal P\sigma|^2+\frac{n}{n-1}\int_{N_\lambda}\langle\mathcal R\sigma,\sigma\rangle\\
		&+\frac{n}{n-1}\int_{\partial N_\lambda}\langle (\overbar c(\overbar \nu)\tdirac+\tcon_{\overbar \nu})\sigma,\sigma\rangle
	\end{split}
\end{equation}

We have the following estimate for the term $\mathcal R$ in line \eqref{eq:D^2} above.

\begin{lemma}\label{lemma:curvature>=}
	If the curvature operator of $(X, g_X)$ is non-negative, then
	\begin{equation}\label{eq:curvature>=}
		\mathcal R\geq \frac{\Sc_{\overbar g}}{4}-\frac{f^*\Sc_{g_X}}{4\varphi(f^*r)^2}.
	\end{equation}
\end{lemma}
\begin{proof}
	For $2$-forms of $N$, we define the Clifford multiplication by
	\begin{equation}\label{eq:c(2-form)}
		\overbar c(\overbar e_i\wedge \overbar e_j)=\overbar c(\overbar e_i)\overbar c(\overbar e_j),
	\end{equation}
	where $\overbar e_i, \overbar e_j \in TN$ are orthogonal. The Clifford  multiplication $c(w)$ for a $2$-form $w$ over $M$ is defined similarly. 
	
	Let $P$ be the orthogonal projection from $TM$ to $TX$.	By the Bochner--Lichnerowicz--Weitzenb\"ock formula, we have
	\begin{equation}\label{eq:BLW}
		\mathcal R =\frac{\Sc_{\overbar g}}{4}-\frac{1}{2}\sum_{i,j}\langle \widehat R (Ph_*\overbar w_j),w_i\rangle_M \, \overbar c(\overbar w_j)\otimes c(w_i),
	\end{equation}
	where $\{\overbar w_j\}$ is a local orthonormal basis of $2$-forms on $N$, and $\{w_i\}$ is a local orthonormal basis of leaf-wise $2$-forms on $M$, and $\widehat R=\varphi^{-2}R_X$ is the leaf-wise curvature operator of $M$, i.e., $\widehat R=\varphi^{-2}R_X$ is  the curvature operator of $(X, \varphi^2 g_X)$. As the curvature operator $\widehat R$ is non-negative along each leaf, there exists a self-adjoint $L\in \Endo(\Bigwedge^2 TX)$  such that $\widehat R=L^2$, that is,
	$\langle \widehat R  w_j,w_i\rangle_{M}=\langle L  w_j,L w_i\rangle_{M}.$
	
	Set $$\overbar L w_k\coloneqq \sum_i\langle L w_k,Ph_*\overbar w_i\rangle_M\overbar w_i\in\Bigwedge^2 TN$$
	and
	$$\alpha=\frac{\varphi(h^*r)}{\varphi(f^*r)}.$$
	The second term on the right hand side  of \eqref{eq:BLW} can be written as
	\begin{align*}
		&-\frac{1}{2}\sum_{i,j}\langle \widehat RPh_*\overbar w_j,w_i\rangle_{M} \overbar c(\overbar w_j)\otimes c(w_i)\\
		=&-\frac{1}{2}\sum_{i,j,k}\langle L(Ph_*\overbar w_j),w_k\rangle_M \cdot 
		\langle Lw_i,w_k\rangle_M \cdot  \overbar c(\overbar w_j)\otimes c(w_i)\\
		=&-\frac{1}{2}\sum_{k}\overbar c(\overbar Lw_k)\otimes c(L w_k)\\
		=&\frac{1}{4}\sum_k\Big(\alpha^{-2}\overbar c(\overbar Lw_k)^2\otimes 1+\alpha^{2}\otimes c(L w_k)^2-\big(\alpha^{-1}\overbar c(\overbar Lw_k)\otimes 1+\alpha\otimes  c(L w_k)\big)^2\Big)\\
		\geq&\frac 1 4 \sum_k \alpha^{-2} \overbar c(\overbar Lw_k)^{2}\otimes 1+\frac{1}{4}\sum_k \alpha^{2}\otimes c(L w_k)^2, 
	\end{align*}
	where the last inequality follows from the fact  that the element
	\[ \alpha^{-1}\overbar c(\overbar Lw_k)\otimes 1+\alpha\otimes  c(L w_k)\] is skew-symmetric, hence its square is non-positive.

	The same proof for the Lichnerowicz formula (cf. \cite[Theorem II.8.8]{spingeometry}) shows that
	$$\alpha^2\sum_k  c(L w_k)^2=-  \alpha^2\frac{h^*\Sc_{\varphi^2g_X}}{2}= - \alpha^2\frac{h^*\Sc_{g_X}}{2 \varphi(h^*r)^2}= - \frac{f^*\Sc_{g_X}}{2\varphi(f^*r)^2},$$
	where by construction we have $f^*\Sc_{g_X}=h^*\Sc_{g_X}$.
	Similarly, by the definition of $\overbar L$, we have
	\begin{align*}
		\sum_k\overbar c(\overbar Lw_k)^2 & =  \sum_{i,j,k}\langle Lw_k, Ph_*\overbar w_i\rangle_M \cdot 
		\langle Lw_k, Ph_*\overbar w_j\rangle_M \cdot \overbar c(\overbar w_i)\otimes c(\overbar w_j) \\
		& =\sum_{i,j}\langle \widehat R(Ph_*\overbar w_i),Ph_*\overbar w_j \rangle_M \cdot \overbar c(\overbar w_i)\overbar c(\overbar w_j).
	\end{align*}
	We choose a local $\overbar g$-orthonormal frame $\overbar e_1,\ldots,\overbar e_n$ of $TN_\lambda$ and a local $g$-orthonormal frame $e_1,\ldots,e_n$ of $TM_\mu$ such that $Ph_*\overbar e_i=\mu_i e_i$ with $\mu_i\geq 0$. This can be seen from the singular value decomposition of the map $Ph_*$. Then we have 
	$Ph_*(\overbar e_i\wedge\overbar e_j)=\mu_i\mu_j e_i\wedge e_j$.
	As $f$ is distance-non-increasing, it follows from the construction of the map $h$ that leafwise $\|dh\|\leq \alpha$. In particular, we have
	$0\leq\mu_i\leq \alpha$ for each $i$. Therefore
	\begin{equation}\label{eq:Rijij}
		\alpha^{-2}\sum_k \overbar c(\overbar Lw_k)^2=- \alpha^{-2}\sum_{i<j}\mu_i^2\mu_j^2(h^\ast  \widehat R_{ijji})\geq -\frac{f^*\Sc_{g_X}}{2\varphi(f^*r)^2}.
	\end{equation}
	This finishes the proof.
\end{proof}

\begin{remark}\label{rk:area}
	To deduce Lemma \ref{lemma:curvature>=}, one may relax the condition that $f$ is distance-non-increasing to that $Pf_*\colon TN\to TX$ is area-non-increasing. Indeed, in this case, the singular value decomposition of $Ph_\ast$ in the proof of Lemma \ref{lemma:curvature>=} implies  that $0\leq \mu_i\mu_j\leq \alpha^2$ for each $i<j$. As a consequence, we see that the inequality in line \eqref{eq:Rijij} still holds.
\end{remark}

Since $\sigma$ satisfies the boundary condition $B$, by using the fact that $c(\nu)=\pm c(\partial_r)$ is parallel with respect to $\tcon$, a standard computation shows that 
\begin{equation}\label{eq:A>=}
	\int_{\partial N_\lambda}\langle (\overbar c(\overbar \nu)\tdirac+\tcon_{\overbar \nu})\sigma,\sigma\rangle=\frac 1 2
	\int_{\partial N_\lambda}\langle H_{\overbar g}\sigma,\sigma\rangle,
\end{equation}
where $H_{\overbar g}$ is the mean curvature of $\partial N_\lambda$, cf. \cite[Lemma 2.9]{Wang:2023warp}. 

To summarize, we have
\begin{equation}\label{eq:D^2-2}
		\begin{split}
		\int_{N_\lambda} |\tdirac\sigma|^2\geq &\frac{n}{n-1}\int_{N_\lambda} |\mathcal P\sigma|^2+
		\frac{n}{n-1}\int_{N_\lambda}\Big(\frac{\Sc_{\overbar g}}{4}-\frac{f^*\Sc_{g_X}}{4\varphi(f^*r)^2}\Big)|\sigma|^2\\
		&+\frac{n}{n-1}\int_{\partial N_\lambda}\frac{H_{\overbar g}}{2}|\sigma|^2
	\end{split}
\end{equation}

Now we consider the second and third terms on the right hand side of the equation from  line \eqref{eq:Bsquare}. By the Stokes formula, we have
\begin{equation}\label{eq:2and3}
	\begin{split}
		&\int_{N_\lambda} \langle \Psi \sigma,\tdirac\sigma\rangle+\langle\tdirac\sigma,\Psi\sigma\rangle\\
		=&\int_{N_\lambda} \langle\tdirac \Psi \sigma,\sigma\rangle+
		\langle\Psi \tdirac\sigma,\sigma\rangle+
		\int_{\partial N_\lambda} \langle\overbar c(\overbar \nu) \Psi \sigma,\sigma\rangle\\
		=&\int_{N_\lambda} \langle [\tdirac,\Psi] \sigma,\sigma\rangle
		+\int_{\partial N_\lambda} \langle\overbar c(\overbar \nu) \Psi\sigma,\sigma\rangle.
	\end{split}
\end{equation}
Note that
\begin{equation}\label{eq:grad}
	\begin{split}
	[\tdirac,\Psi]=&\frac{n}{2}\overbar c\big(\grad_{\overbar g}(\psi(h^*r))\big) \cdot \mathscr E  c(\partial_r)
	=  \frac{n}{2} \psi'(h^*r) \cdot \overbar c\big(\grad_{\overbar g}(h^*r)\big)\cdot \mathscr Ec(\partial_r)\\
	\geq &\frac{n}{2} \psi'(h^*r)\cdot |\grad_{\overbar g}(h^*r)| \\
	= & \frac{n}{2} \psi'(h^*r)  \rho'(f^*r) |\grad_{\overbar g}(f^*r)|\geq \frac{n}{2} \psi'(h^*r) \rho'(f^*r)
\end{split}
\end{equation}
Here we have used the assumptions that $\psi'=(\log\varphi)''\leq 0$ and $f$ is distance-non-increasing.

For the boundary term in line \eqref{eq:2and3}, we have
\begin{equation*}
	\langle\overbar c(\overbar \nu) \Psi\sigma,\sigma\rangle
	=-\frac{n}{2} \psi(h^*r)\langle\mathscr E\overbar c(\overbar \nu)  c(\partial_r)\sigma,\sigma\rangle.
\end{equation*}
Note that $\log\varphi(r) \to -\infty$ as $r\to \pm c$. Since $\psi'=(\log\varphi)''\leq 0$ and the domain of $\psi$ is a bounded interval, it follows that 
\begin{equation}\label{eq:meanlarge}
\lim _{r\to \pm  c} \psi(r)= \lim _{r\to \pm  c} (\log \varphi)'(r) = \mp\infty.
\end{equation} 
By construction, $h^\ast r = \pm \mu$  on the components of $\partial N_\lambda$.  Consequently, when $\mu$ is sufficiently close to $c$, we have 
\begin{equation}\label{eq:boundary}
	\langle\overbar c(\overbar \nu) \Psi\sigma,\sigma\rangle=
	\frac{n}{2}|\psi(h^\ast r)|\cdot |\sigma|^2
\end{equation}
on all components of $\partial N_\lambda$, since $\sigma$ satisfies the boundary condition $B$ given in  Definition \ref{def:boundaryCondition}.

\begin{proposition}\label{prop:D^2}
	With the notation above, there is some $c_0>0$
	\begin{equation}\label{eq:D^2all}
		\begin{split}
			\|\tdirac_\Psi\sigma\|^2\geq&\frac{n}{n-1}\int_{N_\lambda} |\mathcal P\sigma|^2+\frac{n}{4(n-1)}\int_{N_\lambda}(\Sc_{\overbar g}-f^*\Sc_{g})|\sigma|^2\\
			&-\frac{\varepsilon' n}{2}\int_{f^{-1}(\mathcal N_\varepsilon(I_0)\times X)}(|\psi'(h^*r)|+c_0)|\sigma|^2\\
			&+\int_{\partial N_\lambda}\big(\frac{n}{2(n-1)}H_{\overbar g}+
			\frac{n}{2}|\psi(h^*r)|\big)|\sigma|^2
		\end{split}
	\end{equation}
	for any smooth section $\sigma$ of $E$ over $N_\lambda$ satisfying the boundary condition $B$, where $c_0$ is independent of $\varepsilon,\varepsilon'$ and $\lambda$.
\end{proposition}
\begin{proof}
	By applying line \eqref{eq:D^2-2}, \eqref{eq:2and3},  \eqref{eq:grad}and \eqref{eq:boundary} to line \eqref{eq:Bsquare}, we obtain 
	\begin{equation}\label{eq:estimate}
		\begin{split}
			&\|\tdirac_\Psi\sigma\|^2\geq\frac{n}{n-1}\int_{N_\lambda} |\mathcal P\sigma|^2\\
			&+\int_{N_\lambda}\Big(\frac{n}{4(n-1)}\Big[\Sc_{\overbar g}-\frac{f^*\Sc_{g_X}}{\varphi(f^*r)^2}\Big]+\frac{n}{2}\psi'(h^*r) \expand'(f^*r)+\frac{n^2}{4}\psi(h^*r)^2\Big)|\sigma|^2\\
			&+\int_{\partial N_\lambda}\big(\frac{n}{2(n-1)}H_{\overbar g}+
			\frac{n}{2}|\psi(h^*r)|\big)|\sigma|^2
 		\end{split}
	\end{equation}
	 For the warped product metric $g=dr^2+\varphi^2g_X$ on $M$, its scalar curvature is given by the following formula
	\begin{equation}\label{eq:scalar-warp}
		\frac{n}{4(n-1)}\Sc_g=\frac{n}{4(n-1)}\frac{\Sc_{g_X}}{\varphi^2}-\frac{n}{2}\psi'-\frac{n^2}{4}\psi^2.
	\end{equation}
	By our choice of the interval $I_0$ at the beginning of Subsection \ref{sec:estimates},  we have 
	\begin{equation}\label{eq:h>=f}
		\frac{n}{2}\psi'(h^*r)+\frac{n^2}{4}\psi(h^*r)^2 \geq \frac{n}{2}\psi'(f^*r)+\frac{n^2}{4}\psi(f^*r)^2.
	\end{equation}
	on $N_\lambda\backslash f^{-1}(I_0\times X)$, since $\expand(r)$ is closer to $\pm c$ than $r$. Inside $f^{-1}(I_0\times X)$, as the map $h$ is $C^\infty$-close to $f$, line \eqref{eq:h>=f} essentially becomes an equality but up to a small error, which is proportional to $\varepsilon'$. More precisely, we have 
	\begin{equation}\label{eq:h>=f-c0}
		\frac{n}{2}\psi'(h^*r)+\frac{n^2}{4}\psi(h^*r)^2 \geq \frac{n}{2}\psi'(f^*r)+\frac{n^2}{4}\psi(f^*r)^2-\frac{nc_0}{2}\varepsilon'.
	\end{equation}
	on $f^{-1}(I_0\times X)$ for all   sufficiently small $\varepsilon'$, where $c_0>0$ is a positive constant that only depends on the geometry of $I_0\times X$, the geometry of $f^{-1}(I_0\times X)$, and the restriction  of the map $f\colon N\to M$ on $f^{-1}(I_0\times X)$. In particular, $c_0$ is independent of $\varepsilon, \varepsilon'$ and $\lambda$. 
	
	Recall that $1\leq \expand'\leq 1+\varepsilon'$, and furthermore $\expand'=1$ outside $\mathcal N_\varepsilon(I_0)$. Now the proposition easily follows from the above discussion. 
\end{proof}

Now let us prove Proposition \ref{prop:special}. 
\begin{proof}[Proof of Proposition \ref{prop:special}]
Let $\varepsilon=\varepsilon_0$  given in Proposition \ref{prop:special}. Given $\varepsilon_0'$ as in Proposition \ref{prop:special}, we choose $\varepsilon'>0$ such that
\begin{equation}\label{eq:epsilon'}
\varepsilon'\cdot \max_{r\in\mathcal N_{\varepsilon}(I_0)}
\frac{n}{2}(|\psi'(r)|+c_0)\leq \frac{\varepsilon_0'}{2}\cdot \frac{n}{4(n-1)}, 
\end{equation}
where $c_0$ is the positive constant from Proposition \ref{prop:D^2}. Note that, by such a choice of $\varepsilon'$, the sum of  the second term and the third term on the right hand side of \eqref{eq:D^2all} becomes nonnegative. 

 By construction of the function $\expand$ from line \eqref{eq:expand}, we have
\begin{equation}\label{eq:mu}
	\mu=\expand(\lambda)=\lambda+\kappa(\varepsilon,\varepsilon')
\end{equation}
for all $\lambda$ sufficiently close to $\pm \gamma$.  
Note that
$$\lim_{\lambda\to c-\kappa(\varepsilon,\varepsilon')} \ \inf_{x\in\partial N_\lambda}\big(\frac{n}{2(n-1)}H_{\overbar g}(x)+
\frac{n}{2}|\psi(h^*r)|\big)=+\infty.$$
Therefore, there exists $\lambda<c-\kappa(\varepsilon,\varepsilon')$ such that
\begin{equation}\label{eq:lambda}
	\inf_{x\in\partial N_\lambda}\big(\frac{n}{2(n-1)}H_{\overbar g}(x)+
	\frac{n}{2}|\psi(h^*r)|\big)\geq 1>0
\end{equation}
To summarize, for the $\varepsilon,\varepsilon',\lambda,\mu$  chosen above,  the right-hand side of the inequality  \eqref{eq:D^2all} becomes  nonnegative.

On the other hand, since $N_\lambda$ is compact manifold with boundary and $\widehat D_\Psi$ only differs from the Dirac operator associated to $E=S(TN_\lambda\oplus h^*TM_\mu)$ over $N_\lambda$ by a bounded endomorphism, we have  that $\tdirac_\Psi$ subject to the local boundary condition $B$ is a Fredholm operator, and its Fredholm index equals 
$$\ind(\widehat D_\Psi)=\deg(h)\cdot\chi(M_\mu)$$
where $\chi(M_\mu)$ is the Euler characteristic of $M_\mu = [-\mu, \mu]\times X$. Note that $\chi(M_\mu) = \chi(X)$ and we have   $\deg(h) = \deg(f)$ by construction. Since both $\deg(f)\neq 0$ and $\chi(X)\neq 0$ by assumption, we have 
$$\ind(\widehat D_\Psi)=\deg(h)\cdot\chi(M_\mu) =\deg(f)\cdot \chi(X)\ne 0.$$
It follows that   there is a non-zero section $\sigma$ of $E$ over $N_\lambda$ satisfying the boundary condition $B$ such that $\tdirac_\Psi\sigma=0$. Consequently,  Proposition \ref{prop:D^2},  together with line \eqref{eq:mu} and \eqref{eq:lambda} above, implies  that
	\begin{equation}
	\begin{split}
		0 = \|\tdirac_\Psi\sigma\|^2\geq\frac{n}{n-1}\int_{N_\lambda} |\mathcal P\sigma|^2+\frac{\varepsilon_0'}{2}\frac{n}{4(n-1)}\int_{f^{-1}(\mathcal N_\varepsilon(I_0)\times X)}|\sigma|^2+\int_{\partial N_\lambda}|\sigma|^2
	\end{split}.
\end{equation}
It follows that  $\sigma=0 $ on  $f^{-1}(\mathcal N_\varepsilon(I_0)\times X)$ and $\mathcal P\sigma=0$ on $N_\lambda$. Now $\mathcal P\sigma=0$ and $\tdirac_\Psi\sigma = 0$ imply that   
\begin{equation}
	\widehat\nabla_\xi\sigma-\frac 1 n \overbar c(\xi)\Psi\sigma = \widehat\nabla_\xi\sigma+\frac 1 n \overbar c(\xi)\widehat D\sigma= \mathcal P\sigma =0
\end{equation}
for all $\xi \in TN_\lambda$. 
In particular, along any smooth curve $\Gamma$ in $N_\lambda$, $\sigma$ satisfies the following homogeneous ordinary differential equation
\begin{equation}\label{eq:ODE}
	\widehat\nabla_{\dot\Gamma}\sigma-\frac 1 n \overbar c(\dot\Gamma)\Psi\sigma=0,
\end{equation}
where $\dot\Gamma$ is the tangent vector field of the curve $\Gamma$. 
It follows that $\sigma$ is smooth on $N_\lambda$ and non-zero everywhere. However, we have shown that $\sigma=0 $ on  $f^{-1}(\mathcal N_\varepsilon(I_0)\times X)$. We have arrived at a contradiction. This finishes the proof.
\end{proof}

\subsection{A special case of Theorem \ref{thm:extremality}}\label{sec:special}
In this subsection, we prove the special case of Theorem \ref{thm:extremality} where the leaf $X$ is assumed to have non-vanishing Euler characteristic. 
 
We start with the following Poincar\'e type inequality in Euclidean spaces.
\begin{lemma}\label{lemma:poincareEu}
	Let $I^n=[0,1]^n=I^{n-1}\times[0,1]$ be a cube in $\R^n$ and $\ell$ a positive integer. Denote $K=I^{n-1}\times[0,\ell]$. Let $A$ be a smooth matrix-valued function on $\R^n$ with $\|A\|\leq M$ on $K$. Then for any smooth vector-valued function $\alpha$ on $\R^n$, we have
	\begin{equation}\label{eq:poincareEu}
		\int_{K}|\alpha|^2\leq e^{(2M+1)\ell}\Big(\int_{I^n}|\alpha|^2+\int_K\Big|\big(\frac{d}{dx_n}+A\big)\alpha\Big|^2\Big).
	\end{equation}
\end{lemma}
\begin{proof}
	Set $\beta=\frac{d\alpha}{dx_n}+A\alpha$. We have
	$$\frac{d}{dx_n}|\alpha|^2=2\langle\alpha,\frac{d\alpha}{dx_n}\rangle=2\langle\alpha,\beta\rangle-2\langle\alpha,A\alpha\rangle.$$
	Hence
	$$\frac{d}{dx_n}|\alpha|^2\leq 2|\alpha||\beta|+2M|\alpha|^2\leq(2M+1)|\alpha|^2+|\beta|^2.$$
	It follows that for any $s\in[0,1]$ and $0\leq d\leq \ell-1$, we have
	\begin{equation}\label{eq:Rn}
		\begin{split}
		\int_{I^{n-1}\times\{s+d\}}|\alpha|^2&\leq e^{(2M+1)d}\int_{I^{n-1}\times\{s\}}|\alpha|^2+e^{(2M+1)(s+d)}\int_{I^{n-1}\times[s,s+d]}|\beta|^2\\
		&\leq e^{(2M+1)\ell}\Big(\int_{I^{n-1}\times\{s\}}|\alpha|^2+\int_K|\beta|^2\Big).
	\end{split}
	\end{equation}
Integrating $s$ on $[0,1]$, we obtain
$$\int_{I^{n-1}\times[d,1+d]}|\alpha|^2\leq e^{(2M+1)\ell}\Big(\int_{I^{n-1}\times[0,1]}|\alpha|^2+\int_K|\beta|^2\Big).$$
By a summation for $d=0,1,\ldots,\ell-1$, we obtain
$$\int_{K}|\alpha|^2\leq e^{(2M+1)\ell}\Big(\int_{I^n}|\alpha|^2+\int_K|\beta|^2\Big).$$
This finishes the proof.
\end{proof}

From the Euclidean case, we easily generalize Lemma \ref{lemma:poincareEu} on manifolds.

\begin{lemma}\label{lemma:poincare}
	Let $N$ be an $n$-dimensional Riemannian manifold and $K$ a compact domain in $N$. Let $E$ be a Hermitian vector bundle over $N$ equipped with a connection $\nabla$, which may not preserve the metric. Let $x_0$ be a point in $N$ and $\mathcal N_\delta(x_0)$ the $\delta$-neighborhood of $x_0$. Assume $\mathcal N_\delta(x_0)$ is contained in $K$.  Then there exist $C>0$ such that
	\begin{equation}\label{eq:poincare1}
		\int_K|\sigma|^2\leq C\int_{\mathcal N_\delta(x_0)}|\sigma|^2+C\int_K|\nabla\sigma|^2
	\end{equation}
	for any smooth section $\sigma$ of $E$ over $N$. Here the constants $C$ only depend on $x_0,\delta$ and $K$. 
\end{lemma}
\begin{proof}
	Since $K$ is a compact domain in $N$, there exists a finite cover of $K$ such that each member of the cover is connected to $\mathcal N_\delta(x_0)$ via a compact subspace of $N$ that is diffeomorphic to a Euclidean tube as above. Denote these compact subspaces by $\{\mathcal  T_i\}$. Although the metric of $\mathcal T_i$ is not the Euclidean metric, it is easy to see that the same inequality in Lemma \ref{lemma:poincareEu} applies to sections of $E$ over $\mathcal T_i$, except the constants appearing in line \eqref{eq:poincareEu} need to be replaced by some constants that depend on the geometry of $\mathcal T_i$ and the geometry of $E$ over $\mathcal T_i$.  By summing up the corresponding inequalities over $\mathcal T_i$,  we have 
	$$	\int_K|\sigma|^2\leq C\int_{\mathcal N_\delta(x_0)}|\sigma|^2+C\int_K|\nabla\sigma|^2$$
	for some $C>0$. Here $C$ is a positive constant that only depends on the geometry of $K$, the geometry of $E$ over $K$, and the number of the $\mathcal T_i$'s. This finishes the proof. 

\end{proof}

Now let us prove the following special case of Theorem \ref{thm:extremality}.

\begin{theorem}\label{thm:nonzeroEuler}
	Let $M=(-c,c)\times X$ be an $n$-dimensional manifold equipped with the warped product metric 
	$$g=dr^2+\varphi(r)^2g_X$$
	such that 
	\begin{enumerate}[label=$(\arabic*)$]
		\item $\varphi$ is admissible in the sense of Definition $\ref{def:admissible}$,
		\item the curvature operator of $(X, g_X)$ is  nonnegative, and 
		\item $X$ has nonzero Euler characteristic. 
	\end{enumerate}
	Let $(N,\overbar g)$ be a Riemannian manifold and $f\colon N\to M$  a smooth spin proper map with non-zero degree. If $f$ is distance non-increasing and $\Sc_{\overbar g}\geq f^*\Sc_g$, then $\Sc_{\overbar g}=f^*\Sc_g$.  Furthermore, the following hold. 
	\begin{enumerate}[label=$(\mathrm{\Roman*})$]
		\item If $\varphi$ is strictly log-concave, that is, $(\log\varphi)''<0$, then $N=(-c,c)\times Y$ for some Riemannian manifold $(Y,g_Y)$ and the metric  $\overbar g=dr^2+\varphi(r)^2g_Y$, and the map $f$ respects  the product structures. 
		\item If $\varphi$ is strictly log-concave and the metric $g_X$ on the leaf $X$ has positive Ricci curvature, then $f$ is a local isometry.
	\end{enumerate}
\end{theorem}
\begin{proof}\ 
	
	\textbf{Scalar extremality.} First let us prove the scalar-extremality part of the theorem, that is, we first prove that $\Sc_{\overbar g}=f^*\Sc_g$.
	
	Assume on the contrary that $\Sc_{\overbar g}\ne f^*\Sc_{g}$ somewhere. Then there exist $x_0\in N$ and $\delta>0$ such that 
	\begin{equation}
	\Sc_{\overbar g}(x)\geq f^*\Sc_{g}(x)+\delta,~\forall x\in\mathcal N_{\delta}(x_0).
\end{equation}

	Let $\mathcal P$ be the Penrose operator defined in line \eqref{eq:penrosedef}.
	The exact same proof of Proposition \ref{prop:D^2} shows that for any sufficiently small  $\varepsilon, \varepsilon'>0$, we have 
	\begin{equation}\label{eq:D^2all-2}
		\begin{split}
		\|\tdirac_\Psi\sigma\|^2\geq&\frac{n}{n-1}\int_{N_\lambda} |\mathcal P\sigma|^2+\frac{\delta n}{4(n-1)}\int_{\mathcal N_\delta(x_0)}|\sigma|^2\\
			&-\frac{\varepsilon' n}{2}\int_{f^{-1}(\mathcal N_\varepsilon(I_0)\times X)}(|\psi'(h^*r)|+c_0)|\sigma|^2\\
			&+\int_{\partial N_\lambda}\big(\frac{n}{2(n-1)}H_{\overbar g}+
			\frac{n}{2}|\psi(h^*r)|\big)|\sigma|^2, 
		\end{split}
	\end{equation}
where $c_0$ is the same constant from Proposition \ref{prop:D^2}. 
Since both $\deg(f)\neq 0$ and $\chi(X)\neq 0$ by assumption, we have 
$$\ind(\widehat D_\Psi)=\deg(h)\cdot\chi(M_\mu) =\deg(f)\cdot \chi(X)\ne 0.$$
There exists a non-zero section $\sigma$ of $E= S(TN_\lambda\oplus h^*TM_\mu)$ over $N_\lambda$ satisfying the boundary condition $B$ such that  $\tdirac_\Psi\sigma=0$.  

We define an operator $\mathcal Q\colon C^\infty(N_\lambda,E)\to C^\infty(N_\lambda,T^*N_\lambda\otimes E)$ by
\begin{equation}\label{eq:Q}
	\mathcal Q_\xi\sigma=\tcon_\xi\sigma-\frac{1}{n}\overbar c(\xi)\Psi\sigma,
\end{equation}
which is a connection on $E$ that may not preserve the metric.
Note that $\tdirac_\Psi\sigma=0$ implies  
	\[ \mathcal P\sigma=\mathcal Q\sigma. \] 
	
	If we choose   $\varepsilon, \varepsilon'>0$ to be sufficiently small and $\lambda, \mu$ to satisfy line \eqref{eq:lambda} and \eqref{eq:mu}, then  every term, except the third term,  on the right-hand side of line \eqref{eq:D^2all-2} is nonnegative.  Let $K$ be a fixed compact domain of $N$ such that $K$ contains both  $f^{-1}(\mathcal N_\varepsilon(I_0)\times X)$ and $\mathcal N_\delta(x_0)$. In particular, $K$ does not depend on the choice of $\varepsilon$,  as along as $\varepsilon$ is sufficiently small.  Note that $|\psi'|$ is uniformly bounded on the pre-compact set $\mathcal N_\varepsilon(I_0)$. Without loss of generality, let us say $|\psi'|\leq 1$ on  $\mathcal N_\varepsilon(I_0)$.  
	Now  line \eqref{eq:D^2all-2}, together with the fact that $\tdirac_\Psi\sigma=0$ and  $\mathcal P\sigma=\mathcal Q\sigma$,   implies that  
	\begin{equation}\label{eq:comparison}
		\begin{cases}
			\displaystyle \int_{N_\lambda}|\mathcal Q\sigma|^2\leq \frac{\varepsilon'(n-1)}{2}(1+c_0) \int_K|\sigma|^2,	\vspace{0.2cm}\\
			\displaystyle \delta \int_{\mathcal N_\delta(x_0)}|\sigma|^2\leq 2\varepsilon'(n-1)(1+c_0) \int_K|\sigma|^2.
		\end{cases}
	\end{equation}
Note that $\mathcal Q$ only differs from the spinorial connection $\nabla$ on $E$ by an endomorphism, which is uniformly bounded on $K$. By Lemma \ref{lemma:poincare}, there exists a constant $C>0$ such that
	\begin{equation}\label{eq:poincare}
	\int_K|\sigma|^2\leq C\int_{\mathcal N_\delta(x_0)}|\sigma|^2+C\int_K|\mathcal Q\sigma|^2.
\end{equation}
As we have seen in the proof of Lemma \ref{lemma:poincare}, the constant $C$ only depends on $K$ and the geometry of $E$ over $K$, which is covered by some compact tubes $\mathcal N_{\delta}(x_0)$. Strictly speaking, the metric of $E = S(TN_\lambda\oplus h^*TM_\mu)$, hence its associated Levi-Civita connection, depends on the map $h$. But by construction the map $h$ is $C^\infty$-close to the map $f\colon N\to M$. In particular, the metric and its associated connection of $E$ over the set $K$ are uniformly bounded by some positive constant that is independent of $h$.  Therefore we may choose $C>0$ independent of $\varepsilon,\varepsilon',\lambda$ and $\mu$.

It follows that  we have
$$\int_K|\sigma|^2\leq \varepsilon'(1+c_0)\Big(\frac{2C(n-1)}{\delta}+\frac{C(n-1)}{2}\Big)\cdot \int_K|\sigma|^2.$$
Now we choose $\varepsilon'$ to be sufficiently small so that
$$\varepsilon'(1+c_0)\Big(\frac{2C(n-1)}{\delta}+\frac{C(n-1)}{2}\Big) \leq \frac 1 2<1.$$
It follows that $\sigma\equiv 0$ on $K$. By the first line of \eqref{eq:comparison}, we have $\mathcal Q\sigma=0$ everywhere on $N_\lambda$. Hence $\sigma$ satisfies the homogeneous ordinary differential equation \eqref{eq:ODE} along every curve in $N_\lambda$. In particular, $\sigma$ vanishes everywhere on $N_\lambda$, which contradicts the fact that $\sigma$ is a nonzero section. This proves the scalar extremality part of the theorem. 

\textbf{Scalar rigidity}. Now let us prove the scalar rigidity part of the theorem. 
First, let us prove part (I), that is, we prove that  if $\varphi$ is strictly log-concave, then $(N,\overbar g)$ is also a warped product metric with the same warping function $\varphi$. 

Let us start with the following claim. 
\begin{claim}\label{claim:dr=1} Under the given assumption, we have  
	$|\grad_{\overbar g}(f^*r)|=1.$
\end{claim}
\begin{proof}[Proof of Claim \ref{claim:dr=1}]
	By the assumption that $f$ is distance-non-increasing, we see that 
	\[ |\grad_{\overbar g}(f^*r)|\leq 1. \] Assume on the contrary that 
	\[ |\grad_{\overbar g}(f^*r)|<1 \] somewhere. More precisely, we assume that there exists $x_0\in N$ and $\delta>0$ such that
	\begin{equation}\label{eq:grad>delta}
		|\grad_{\overbar g}(f^*r)(x)|< 1-\delta,~\forall x\in \mathcal N_\delta(x_0).
	\end{equation}
	
	Recall that we have used the inequality \eqref{eq:grad} in the proof of Proposition \ref{prop:D^2}. If we do not apply the inequality \eqref{eq:grad}, then  the inequality \eqref{eq:D^2all} may be written as follows
	\begin{equation}\label{eq:D^2all-3}
		\begin{split}
			\|\tdirac_\Psi\sigma\|^2\geq&\frac{n}{n-1}\int_{N_\lambda} |\mathcal P\sigma|^2+\frac{n}{4(n-1)}\int_{N_\lambda}(\Sc_{\overbar g}-f^*\Sc_{g})|\sigma|^2\\
			&-\frac{\varepsilon' n}{2}\int_{f^{-1}(\mathcal N_\varepsilon(I_0)\times X)}(|\psi'(h^*r)|+c_0)|\sigma|^2\\
			&+\frac{n}{2}\int_{N_\lambda}|\psi'(h^*r)|\cdot\big\langle\big(\expand'(f^*r)-\overbar c(\grad_{\overbar g}(h^*r))\mathscr E c(\partial_r)\big)\sigma,\, \sigma\big\rangle\\
			&+\int_{\partial N_\lambda}\big(\frac{n}{2(n-1)}H_{\overbar g}+
			\frac{n}{2}|\psi(h^*r)|\big)|\sigma|^2.
		\end{split}
	\end{equation}
In other words,  the effect of applying the inequality \eqref{eq:grad} is to eliminate the third line from the above inequality \eqref{eq:D^2all-3}. In particular, the above inequality \eqref{eq:D^2all-3} becomes  the inequality  \eqref{eq:D^2all} if we apply  the inequality \eqref{eq:grad}. 
	
	Note that 
	\[ |c(\grad_{\overbar g}(h^*r))\mathscr E c(\partial_r)| = |c(\grad_{\overbar g}(h^*r))| =  \rho'(f^\ast r) |c(\grad_{\overbar g}(f^*r))| \leq \rho'(f^\ast r) \]
	where Here we have used the fact that  $f$ is distance-non-increasing. Therefore, if we replace the domain of integral $N_\lambda$ in the third line of the inequality \eqref{eq:D^2all-3} with $\mathcal N_\delta(x_0)$, the  inequality \eqref{eq:D^2all-3} still holds. 
	Since $\varphi$ is strictly log-concave, there exists $\delta'>0$ that
	$$|\psi'(f^*r)|>\delta',~\forall x\in \mathcal N_{2\delta}(x_0).$$
It follows from line \eqref{eq:grad>delta} that for $\varepsilon$ and $\varepsilon'$  sufficiently small, we have 
	\begin{equation}\label{eq:D^2all-4}
		\begin{split}
			\|\tdirac_\Psi\sigma\|^2\geq&\frac{n}{n-1}\int_{N_\lambda} |\mathcal P\sigma|^2+\frac{n}{4(n-1)}\int_{N_\lambda}(\Sc_{\overbar g}-h^*\Sc_{g})|\sigma|^2\\
			&-\frac{\varepsilon' n}{2}\int_{f^{-1}(\mathcal N_\varepsilon(I_0)\times X)}(|\psi'(h^*r)|+c_0)|\sigma|^2
			+\frac{n\delta\delta'}{2}\int_{\mathcal N_\delta(x_0)}|\sigma|^2\\
			&+\int_{\partial N_\lambda}\big(\frac{n}{2(n-1)}H_{\overbar g}+
			\frac{n}{2}|\psi(h^*r)|\big)|\sigma|^2.
		\end{split}
	\end{equation}
By using Lemma \ref{lemma:poincare},  the same argument as in the proof for  the \textbf{Scalar extremality} part above shows that the inequality \eqref{eq:D^2all-4} leads to a contradiction. This finishes the proof of the claim. 
\end{proof}

By Claim \ref{claim:dr=1}, we see that $f^\ast r$ is a smooth function on $N$ with no critical points, and moreover $f_*(\grad_{\overbar g}(f^*r))=\partial_r$. Therefore, $N$ is diffeomorphic to $(-c,c)\times Y$, where $Y$ is the preimage of the leaf $X_{r_0}$ for some (hence any) $r_0$. Furthermore, it follows that  $f\colon N \to M=(-c,c)\times X$ preserves the product structures. The metric $\overbar g$ is now given by
$$\overbar g=dr^2+\overbar g_r.$$
We shall show that $\overbar g$ is indeed a warped product metric.

For notational simplicity, we also denote  $\grad_{\overbar g}(f^*r)$ by $\partial_{\overbar r}$. We define a tensor field $V\in C^\infty(N,TN\otimes TN)$ by
\begin{equation}\label{eq:V}
	V_\xi\coloneqq \nabla^N_\xi \partial_{\overbar r}-\psi(f^*r)\big(\xi-\langle\xi,\partial_{\overbar r}\rangle_N\cdot \partial_{\overbar r}\big)
\end{equation}
for any tangent vector field $\xi$ of $N$. 
\begin{claim}\label{claim:vanish} 
	We have $V\equiv 0.$
\end{claim}
\begin{proof}[Proof of Claim \ref{claim:vanish}]
	Assume on the contrary that $V\ne 0$. Then there exists $x_0\in N$, $\delta>0$, and a unit tangent vector field $\xi$ over $\mathcal N_{\delta}(x_0)$ such that 
	\begin{equation}\label{eq:V>delta}
		|(V_\xi)_x|>\delta,~\forall x\in \mathcal N_{\delta}(x_0).
	\end{equation}
	 By compactness, we assume that
	$$|V|\leq C_1,\text{ and }|\nabla V|\leq C_1$$
	for some $C_1>0$ on $\mathcal N_\delta(x_0)$. 
	
Let us first prove a technical estimate (Inequality \eqref{eq:Ndelta} below), which will then be combined with Lemma \ref{lemma:poincare} to get a contradiction. Suppose $\sigma$ is a smooth section of $E$ over $N_\lambda$. We set
	\begin{equation}
		T=1-\overbar c(\partial_{\overbar r})\mathscr E c(\partial_r).
	\end{equation}
	A direct computation shows that
	\begin{equation}
		[\mathcal Q_\xi,T]=\overbar c(V_\xi)\mathscr E c(\partial_r),
	\end{equation}
	where the operator $\mathcal Q$ is defined in line \eqref{eq:Q}.
	We may assume without loss of generality that $\delta$ is sufficiently small so that the boundary $\partial\mathcal N_s(x_0)$ of $\mathcal N_s(x_0)$ is a smooth hypersurface in $\mathcal N_\delta(x_0)$ for each $0<s\leq \delta$. By the Stokes formula, we have
	\begin{equation*}
		\begin{split}
			&\int_{\mathcal N_s(x_0)}|\overbar c(V_\xi)\sigma|^2=\int_{\mathcal N_s(x_0)}\langle (\mathcal Q_\xi T-T\mathcal Q_\xi)\sigma,[\mathcal Q_\xi,T]\sigma\rangle\\
			=&\int_{\mathcal N_s(x_0)}\big(\langle\mathcal Q_\xi T\sigma,[\mathcal Q_\xi,T]\sigma\rangle-\langle\mathcal Q_\xi\sigma,T[\mathcal Q_\xi,T]\sigma\rangle \big)\\
			=&\int_{\mathcal N_s(x_0)}\Big(\langle T\sigma,[\mathcal Q_\xi,[\mathcal Q_\xi,T]]\sigma+[\mathcal Q_\xi,T]\mathcal Q_\xi\sigma\rangle-\langle\mathcal Q_\xi\sigma,T[\mathcal Q_\xi,T]\sigma\rangle \Big)\\
			&+\int_{\partial \mathcal N_s(x_0)}\langle \nu_s,\xi\rangle
			\langle T\sigma,[\mathcal Q_\xi,T]\sigma\rangle,
		\end{split}
	\end{equation*}
where $\nu_s$ is the unit inner normal vector of $\partial \mathcal N_s(x_0)$.
Note that
\begin{equation}
	\begin{split}
	[\mathcal Q_\xi,[\mathcal Q_\xi,T]]=[\mathcal Q_\xi,\overbar c(V_\xi)\mathscr E c(\partial_r)]
	=\overbar c(\nabla^N_{\xi}V_\xi)\mathscr E c(\partial_r)-\psi(h^*r)\overbar c(\xi\wedge V_\xi).
	\end{split}
\end{equation}
	Therefore, there is $C_2>0$ such that
	\begin{equation}\label{eq:Ns}
		\begin{split}
			\int_{\mathcal N_s(x_0)}|\overbar c(V_\xi)\sigma|^2
			\leq& C_2\int_{\mathcal N_s(x_0)}\big(|T\sigma||\sigma|+|\mathcal Q_\xi\sigma||\sigma|+|\mathcal Q_\xi\sigma||T\sigma|\big)\\
			&+C_2\int_{\partial\mathcal N_s(x_0)}|T\sigma||\sigma|
		\end{split}
	\end{equation}
	For $x\in\mathcal N_\delta(x_0)$, we define $F(x)=\delta-\dist(x,x_0)$,
	which gives  a continuous positive function on $\mathcal N_\delta(x_0)$. 
	We integrate both sides of inequality \eqref{eq:Ns} for $s\in [0,\delta]$. By changing the order of integration, we obtain
	\begin{equation*}
		\begin{split}
			\int_{\mathcal N_\delta(x_0)}F\cdot|\overbar c(V_\xi)\sigma|^2
			\leq& C_2\int_{\mathcal N_\delta(x_0)}F\cdot\big(|T\sigma||\sigma|+|\mathcal Q_\xi\sigma||\sigma|+|\mathcal Q_\xi\sigma||T\sigma|\big)\\
			&+C_2\int_{\mathcal N_\delta(x_0)}|T\sigma||\sigma|.
		\end{split}
	\end{equation*}
	Therefore, there exists $C_3>0$ such that
	\begin{equation}\label{eq:Ndelta}
		\begin{split}
			\int_{\mathcal N_{\frac{\delta}{2}}(x_0)}|\overbar c(V_\xi)\sigma|^2
			\leq C_3\int_{\mathcal N_\delta(x_0)}\big(|T\sigma||\sigma|+|\mathcal Q_\xi\sigma||\sigma|+|\mathcal Q_\xi\sigma||T\sigma|\big)\\
		\end{split}
	\end{equation}
	
	Now suppose that $\sigma$ is a non-zero section of $E$ over $N_\lambda$ satisfying the boundary condition such that $\tdirac_\Psi\sigma=0$. Let $K$ be a fixed compact domain in $N_\lambda$ that contains $\mathcal N_\delta(x_0)$ and $f^{-1}(\mathcal N_\varepsilon(I_0)\times X)$. By Claim \ref{claim:dr=1}, we see the integrand in the third line of the inequality \eqref{eq:D^2all-3} becomes 
\begin{align*}
	& |\psi'(h^*r)|\cdot\big\langle\big(\expand'(f^*r)-\overbar c(\grad_{\overbar g}(h^*r))\mathscr E c(\partial_r)\big)\sigma,\, \sigma\big\rangle \\
	= & |\psi'(h^*r)| \cdot\big\langle\big( \expand'(f^*r) - \expand'(f^*r)   \overbar c(\grad_{\overbar g}(f^*r))\mathscr E c(\partial_r)\big)\sigma,\, \sigma\big\rangle \\
	 = & |\psi'(h^*r)| \expand'(f^*r)\cdot\big\langle\big( 1  - \overbar c(\partial_{\overbar r})\mathscr E c(\partial_r)\big)\sigma,\, \sigma\big\rangle \\
	 = & |\psi'(h^*r)| \expand'(f^*r)\cdot\langle T\sigma,\, \sigma\rangle
\end{align*}
It follows from the inequality \eqref{eq:D^2all-3} and $\mathcal P\sigma  = \mathcal Q\sigma $ that there is $C_4>0$ such that
	\begin{equation}\label{eq:comparison2}
		\begin{cases}
			\displaystyle \int_{N_\lambda}|\mathcal Q\sigma|^2\leq C_4\varepsilon'\int_K|\sigma|^2,	\vspace{0.2cm}\\
			\displaystyle  \int_{N_\lambda}\langle T\sigma,\sigma\rangle\leq C_4\varepsilon'\int_K|\sigma|^2.
		\end{cases}
	\end{equation}
	Note that $T$ is a self-adjoint endormorphism and $(1-T)^2=1$. It follows that,  at each point of $N$, the operator $T$ is a self-adjoint matrix  with eigenvalues $0$ and $2$. In particular, we have
	\begin{equation}\label{eq:T=T^2}
		\langle T\sigma,\sigma\rangle=\frac 1 2 |T\sigma|^2.
	\end{equation}
	From line \eqref{eq:Ndelta}, \eqref{eq:comparison2}, \eqref{eq:T=T^2} and the Cauchy--Schwarz inequality, we have
	\begin{equation}
		\int_{\mathcal N_{\frac{\delta}{2}}(x_0)}|\overbar c(V_\xi)\sigma|^2
		\leq C_5\sqrt{\varepsilon'}\int_K|\sigma|^2
	\end{equation}
	for some $C_5>0$. Combined with  line \eqref{eq:V>delta}, we obtain 
	\begin{equation}\label{eq:contra-inequal}
		\int_{\mathcal N_{\frac{\delta}{2}}(x_0)}|\sigma|^2
		\leq \frac{C_5\sqrt{\varepsilon'}}{\delta^2}\int_K|\sigma|^2.
	\end{equation}
	We emphasize that the constants $\{C_i\}_{1\leq i \leq 5}$ above  are independent of the choice of the parameters $\varepsilon,\varepsilon',\lambda,\mu$ that appear in the construction of the map $h$. By using Lemma \ref{lemma:poincare},   we see that  the inequality \eqref{eq:contra-inequal},  together with line \eqref{eq:comparison2},  leads to a contradiction (cf. the corresponding argument in the proof for  the \textbf{Scalar extremality} part above). This prove Claim \ref{claim:vanish}.
\end{proof}

Now that we know $V=0$, that is, 
\begin{equation}
	\nabla^N_\xi \partial_{\overbar r}=\psi(f^*r)\big(\xi-\langle\xi, \partial_{\overbar r}\rangle_N\cdot  \partial_{\overbar r}\big), 
\end{equation}
it follows that  the integral curves  of $ \partial_{\overbar r}$ on $N$ are geodesic and the second fundamental form of the leaf $\{r\}\times Y$ is equal to $\psi \cdot I$, where $I$ stands for the identity matrix. In other words, all principal curvatures of $\{r\}\times Y$ are equal to $\psi$.   Therefore,  $(N, \overbar g)$ is also a warped product metric. Moreover, since $\psi = \varphi'/\varphi$,  it follows that $\overbar g$ is  of the form
\[ \overbar g = dr^2 + \varphi^2 g_Y \] 
 where $g_Y$ is some Riemannian metric on $Y$. 
This proves \textbf{Scalar rigidity} part (I).

Now we prove \textbf{Scalar rigidity} part (II). Denote $Y_r=\{r\}\times Y$. By the proof of \textbf{Scalar rigidity} part (I) above, the map $f$ maps $Y_r$ to $X_r$. By assumption, $X$ is Ricci positive and $f\colon Y_r\to X_r$ is distance-non-increasing. To prove \textbf{Scalar rigidity} part (II), it suffices to show that 
	$f_*\colon TY_r\to TX_r$ is an isometry for every $r\in (-c, c)$.

	Assume to the contrary that $f_\ast$ is not an isometry for some $r_0$ and $y_0\in Y_{r_0}$.  Consider the singular value decomposition of $f_*\colon T_{y_0}Y_{r_0}\to T_{f(y_0)}X_{r_0}$, that is, there exist orthonormal bases $\{\overbar e_i\}_{1\leq i\leq n-1}$ of $T_{y_0}Y_{r_0}$ and $\{e_i\}_{1\leq i\leq n-1}$ of $T_{f(y_0)}X_{r_0}$ such that $f_*\overbar e_i=\mu_ie_i$ for some $\mu_i\in[0,1]$. By our assumption,  there exists some $i_0$ such that $$\mu_{i_0}\leq \sqrt{1-\delta'}$$
	for some $\delta'>0$. Then by definition of the map $h$, we have $h_*\overbar e_i=\alpha\mu_iv_i$, where $\{v_i\}_{1\leq i\leq n-1}$ is an orthonormal basis of $T_{h(y_0)}X_{\expand(r_0)}$ and 
	$\alpha = \varphi(\expand(r_0))/\varphi(r_0) $. Compare with the proof of Lemma \ref{lemma:curvature>=} to see how the constant $\alpha$ enters into the estimates. It follows from the above discussion that the inequality \eqref{eq:Rijij} becomes a strict inequality. More precisely, at the point $(r_0, y_0)\in N$, we have 
	\begin{equation}\label{eq:Rijij-strict}
		\alpha^{-2}\sum_k \overbar c(\overbar Lw_k)^2=- \alpha^{2}\sum_{i<j}\mu_i^2\mu_j^2(h^\ast  \widehat R_{ijji})\geq-\frac{f^*\Sc_{g_X}}{2\varphi(f^*r)^2}+\delta'\cdot \frac{f^*\Ric_{X}(e_{i_0})}{\varphi(f^*r)^2}
	\end{equation}
where $\Ric_X(e_{i_0})$ is the Ricci curvature of $(X, g_X)$ at $f(y_0)$ in  the direction of $e_{i_0}$. By continuity, the inequality \eqref{eq:Rijij-strict} above (but with possibly a smaller $\delta'$) also holds on 
a small neighborhood of $(r_0,y_0)$ in $N = (-c, c) \times Y$. 
To summarize, we see that  there exist $\delta>0$ and $x_0 = (r_0, y_0)\in N$ such that
	\begin{equation}\label{eq:curvature>}
		\mathcal R  \geq \frac{\Sc_{\overbar g}}{4}-\frac{f^*\Sc_{g_X}}{4\varphi(f^*r)^2}+\delta
	\end{equation}
	on $\mathcal N_\delta(x_0)$,  where $\mathcal R$ is the curvature term from line \eqref{eq:D^2} (cf. the proof of Lemma \ref{lemma:curvature>=}). Together with the estimates from the proof of Proposition \ref{prop:D^2}, this  implies that there exists $x_0\in N$ and $\delta>0$ such that 
\begin{align*}
& \frac{n}{n-1}	\mathcal R +  \frac{n}{2}\psi'(h^*r) \expand'(f^*r)+\frac{n^2}{4}\psi(h^*r)^2 \\
& \geq 
\frac{n}{4(n-1)} \left(\Sc_{\overbar g}(x) -  f^*\Sc_{g}(x)\right)+\delta \geq \delta,\quad \forall x\in\mathcal N_{\delta}(x_0). 
\end{align*} Now we proceed exactly the same way as the proof of \textbf{Scalar extremality} part and arrive at a contradiction. This proves \textbf{Scalar rigidity} part (II), hence completes the proof of the theorem. 
\end{proof}

\section{Scalar curvature rigidity of degenerate spherical bands}\label{sec:sphere}
In the previous section, we prove a special case of Theorem \ref{thm:extremality} where the leaf $X$ of $M = (-c, c)\times X$ has non-zero Euler characteristic. In this section, we shall  prove a special case of  Theorem \ref{thm:extremality} where the leaf $X$ is a standard round sphere. 

\begin{theorem}\label{thm:extremalitySphere}
	Let $M=(-c,c)\times X$ be an $n$-dimensional manifold equipped with the warped product metric 
	$$g=dr^2+\varphi(r)^2g_X$$
	such that 
	\begin{enumerate}[label=$(\arabic*)$]
		\item $\varphi$ is admissible in the sense of Definition $\ref{def:admissible}$ and 
		\item $(X,g_X)$ is the $(n-1)$-dimensional standard round sphere $(\sph^{n-1},g^{\sph^{n-1}}_{st})$. 
	\end{enumerate}
	Let $(N,\overbar g)$ be a spin Riemannian manifold and $f\colon N\to M$ be a smooth proper map with non-zero degree. If $f$ is distance non-increasing and $\Sc_{\overbar g}\geq f^*\Sc_g$, then $\Sc_{\overbar g}=f^*\Sc_g$. Furthermore, if in addition $n\geq 3$ and $\varphi$ is strictly log-concave, then $f$ is a  isometry.
\end{theorem}

Of course, the case where $X$ is an even dimensional sphere has already been covered by Theorem \ref{thm:nonzeroEuler}. So it remains to consider the case where $X$ is a standard odd dimensional sphere.  In order to overcome the difficulty caused by the fact an odd dimensional sphere has vanishing Euler characteristic,  we follow Llarull's idea and take the  direct product with a large circle. But this introduces an extra small error term in the relevant curvature estimates.  A key step of our proof is to dominate this extra error term by  the Poincar\'e type inequality from Lemma \ref{lemma:poincare}.  As mentioned in the introduction, due to the extra error term caused by introducing an auxiliary circle, there is a minor gap in Llarull's  proof  for the scalar rigidity of a \emph{closed} standard \emph{odd} dimensional sphere \cite[Section 4]{Llarull}. In order to make more clear how the Poincar\'e type inequality (Lemma \ref{lemma:poincare}) enters into our proof of Theorem \ref{thm:extremalitySphere}, let us first demonstrate how it can applied to fix this minor gap in Llarull's proof  for the scalar rigidity of a closed standard odd dimensional sphere. 

\begin{theorem}[{Llarull \cite{Llarull}}]\label{thm:llarull-odd}
	Let $(\mathbb S^{2k+1}, g_{st})$ be the standard round unit sphere of dimension $(2k+1)\geq 3$. Let $(N,\overbar g)$ be a closed spin Riemannian manifold and  $f\colon N\to M$  a smooth map with non-zero degree. If  $\Sc_{\overbar g}\geq 2k(2k+1)$ and $f$ is area-non-increasing,  then $f$ is an isometry.
\end{theorem}
\begin{proof}
	We first follow closely Llarull's original proof. Let $N\times \mathbb S^1_R$ be the Riemannian product of $N$ with the circle  $\mathbb S^1_R$ of radius $R$. Consider the following map
	\begin{equation}
			(N\times \sph^1_R,\overbar g+R^2d\theta^2)\xlongrightarrow{f\times\frac{\id}{R}}(\sph^{2k+1}\times \sph^1, g_{st}+d\theta^2)\xlongrightarrow{\alpha} 
			\sph^{2k+1}\wedge \sph^1 = \sph^{2k+2}
	\end{equation} 
	where  $\mathbb S^{2k+2}$ is the standard unit round sphere, $f\times\frac{\id}{R}$ is given by $(f\times\frac{\id}{R})(x, \theta) = (x, \frac{\theta}{R})$ for all $(x, \theta) \in N\times \sph^1_{R}$, and $\alpha$ is a distance-non-increasing map of nonzero degree. 
	
	Let us write $\widetilde f = \alpha\circ (f\times \frac{\id}{R})$.  Let $E$ be the following spinor bundle over $N\times \sph^1_R$: 
	\[ E = S\left(T(N\times \sph^1_R) \oplus \widetilde f^\ast T\sph^{2k+2}\right). \]  
	Let $D$ be the corresponding Dirac operator for $E$. By the Bochner--Lichnerowicz--Weitzenbock formula, we have 
	\begin{equation}\label{eq:D^2-sphere}
		D^2=\nabla^*\nabla+\frac{\Sc_{\overbar g}}{4} +\frac{1}{8}\sum_{i,j}\sum_{k,l}\langle \widetilde f^\ast R_{\overbar e_i,\overbar e_j}e_k,e_l\rangle \, \overbar c(\overbar e_i)\overbar c(\overbar e_j)\otimes c(e_k)c(e_l),
	\end{equation}
	where $\{e_i\}$ is a local orthonormal basis of $\widetilde f^\ast T\sph^{2k+2}$, and  $\widetilde f^\ast R$ is the curvature form of $\widetilde f^\ast T\sph^{2k+2}$. Here we have used the obvious fact the scalar curvature of $N\times \sph^1_R$ coincides with the scalar curvature $\Sc_{\overbar g}$ of $(N, \overbar g)$. If  $\{\overbar w_j\}$ and $\{w_i\}$ are local orthonormal bases of $\Bigwedge^2T(N\times \sph^1_R)$ and $\widetilde f^\ast \Bigwedge^2T\sph^{2k+2}$ respectively, 
	then we can rewrite \eqref{eq:D^2-sphere} as
	\begin{equation}\label{eq:D^2-sphere2}
		D^2=\nabla^*\nabla+\frac{\Sc_{\overbar g}}{4}-\frac{1}{2}\sum_{k,l}\langle f_*\overbar w_k,w_l\rangle \, \overbar c(\overbar w_k)\otimes c(w_l).
	\end{equation}
  We choose a local $\overbar g$-orthonormal frame $\overbar e_1,\ldots,\overbar e_{2k+2}$ of $T(N\times \sph^1_R)$, where $\overbar e_{2k+2}$ is tangential to $\sph_R^1$,   and a local $g$-orthonormal frame $e_1,\ldots,e_{2k+2}$ of $T\sph^{2k+2}$ such that $\widetilde f_*\overbar e_i=\mu_i e_i$ with $\mu_i\geq 0$.   Then we have 
  $\widetilde f_*(\overbar e_i\wedge\overbar e_j)=\mu_i\mu_j e_i\wedge e_j$. 
  As $\widetilde f$ is area-non-increasing, we have $\mu_i\mu_j\leq 1$ for all $i\neq j$. Moreover, since $\widetilde f$ is $\frac{1}{R}$-contracting along the $\sph_R^1$ direction, we have $\mu_{2k+2}\leq $$\frac{1}{R}$. It follows that      
  \[  D^2 \geq  \nabla^*\nabla + \frac{\Sc_{\overbar g}}{4} - \frac{2k(2k+1)}{4} - \frac{2k+1}{2R}. \]
In particular, we have 
  \begin{equation}\label{eq:D^2witherror}
  	\|D\varphi\|^2  \geq \|\nabla \varphi\|^2 + \int_{N\times \sph^1_R} \Big(\frac{\Sc_{\overbar g}}{4} - \frac{2k(2k+1)}{4} - \frac{2k+1}{2R}\Big) |\varphi|^2 
  \end{equation}
  for all smooth sections $\varphi$ of $E$ over $N\times \sph^1_R$. So far, we have essentially followed the same argument of Llarull \cite[Section 4]{Llarull}. Note that the Fredholm index of $D$ is nonzero, in fact, equal to $2$ times the degree of $\widetilde f$, where $2$ comes from the Euler characteristic of $\sph^{2k+2}$. Therefore there exists a nonzero section $\sigma$ of $E$ such that $D\sigma =0$. One would like to plug $\sigma$ into  the inequality \eqref{eq:D^2witherror} to conclude that  $\Sc_{\overbar g} =  2k(2k+1)$. However,  a priori,  the extra error term $-\frac{2k+1}{2R}$ prevents us from directly making such a conclusion. In the following, we shall use the Poincar\'e type inequality from Lemma \ref{lemma:poincare} to get round this issue. 
  
  Let us prove  $\Sc_{\overbar g}=2k(2k+1)$ by contradiction. Assume to the contrary that the inequality $\Sc_{\overbar g}\geq 2k(2k+1)$ is strict somewhere. Then there are $x_0\in N$ and $\delta>0$ such that
  $$\Sc_{\overbar g}(x)\geq 2k(2k+1)+\delta,~\forall x\in\mathcal N_{\delta}(x_0).$$
  It follows that 
  \[ \Sc_{N\times \sph^1_R}(x) \geq 2k(2k+1) +\delta, ~\forall x\in\mathcal N_{\delta}(x_0) \times \sph^1_R. \]
  
  We recall that the constants appearing in  the Poincar\'e type inequality from Lemma \ref{lemma:poincare} only depends on the local geometry of $N\times \sph^1_R$ and the bundle $E$,  the number of tubes (as in the proof of Lemma \ref{lemma:poincare}) that cover $N\times\sph^1_R$ and their sizes.
  Let $\widetilde f_1$ be the map $N\times\sph^1_1\to\sph^{2k+2}$ for $R=1$, and $i_R$ 
  $$i_R\colon (N\times \sph^1_R,\overbar g+R^2d\theta^2)\to (N\times \sph^1_1,\overbar g+d\theta^2)$$
  the identity map (at the level of sets). 
 Let 
 	\[ E_1 = S\left(T(N\times \sph^1_1) \oplus \widetilde f_1^\ast T\sph^{2k+2}\right). \]  
 We notice that $E=i_R^*E_1$, and the spinorial connection $\nabla$ on $E$ is the pull-back of the connection on $E_1$ by $i_R$. Clearly $E_1$ is independent of $R$, and as $R\to\infty$, the connection on $i_R^*E_1$ becomes  flatter in the $\sph^1$-direction. Thus, the local geometric data of $N\times \sph^1_R$ and the bundle $E$ over $N\times \sph^1_R$ are uniformly bounded for any $R\geq 1$.
  
  Given $\mathcal N_\delta(x_0)\subset N$, there exists finitely many tubes $\mathcal T_i$ as in the proof of Lemma \ref{lemma:poincare} that contain $\mathcal N_\delta(x_0)$ and cover $N$. Then each tube $\mathcal T_i\times \sph^1_R$  contains $\mathcal N_\delta(x_0)\times\sph^1_R$ and together they cover $N\times\sph^1_R$. The cardinality of the set $\{\mathcal T_i\times \sph^1_R \}$ is clearly independent of $R$. Moreover, the constant  appearing in the corresponding  the Poincar\'e type inequality (as in Lemma \ref{lemma:poincare}) may be chosen so that it only depends on  the size of  $\mathcal T_i$, in particular, is independent of $R$, as long as $R\geq 1$.   Therefore, it follows from Lemma \ref{lemma:poincare} that there exists $C>0$ (independent of $R\geq 1$) such that 
 \begin{equation}\label{eq:poincare3}
 	\int_{N\times \sph^1_R} |\varphi|^2\leq C\int_{\mathcal N_\delta(x_0)\times \sph^1_R} |\varphi|^2+C\int_{N\times \sph^1_R}  |\nabla\varphi|^2
 \end{equation}
for all smooth sections $\varphi$ of $E$.

 Since $D\sigma= 0$, it follows that 
 \begin{align}
 		0 = \|D\sigma\|^2  \geq& \int_{N\times \sph^1_R}  |\nabla\sigma|^2 + \int_{N\times \sph^1_R} \Big(\frac{\Sc_{\overbar g}}{4} - \frac{2k(2k+1)}{4} - \frac{2k+1}{2R}\Big) |\sigma|^2 \notag \\
 	 \geq& \int_{N\times \sph^1_R}  |\nabla\sigma|^2 + \int_{N_{\delta}(x_0) \times \sph^1_R} (\delta - \frac{2k+1}{2R}) |\sigma|^2 -  \frac{2k+1}{2R}\int_{N\times \sph^1_R} |\sigma|^2 \notag \\
 	 \geq& \Big(1-\frac{(2k+1)C}{R}
 	 \Big) \int_{N\times \sph^1_R}  |\nabla\sigma|^2 \label{eq:stokes}\\
 	 &+\Big(\delta - \frac{2k+1}{2R}-\frac{(2k+1)C}{2R}\Big) \int_{N_{\delta}(x_0) \times \sph^1_R}  |\sigma|^2 .\notag 
 \end{align} 
Since $C$ is independent of $R$, for a given sufficiently large $R$, the above estimates imply that  $\sigma$  vanishes on $N_{\delta}(x_0) \times \sph^1_R$ and $\nabla\sigma$ vanishes on $N\times\sph^1_R$. This together with the inequality \eqref{eq:poincare3} implies that $\sigma$ vanishes on $N\times\sph^1_R$, which leads to a contradiction. Therefore, we have proved that $\Sc_{\overbar g}=2k(2k+1)$. 

Now since  $2k+1\geq 3$,  the proof of  the \textbf{Scalar rigidity} part (II) of Theorem \ref{thm:nonzeroEuler} can be easily  adapted to the current setting to show that $f$ is an isometry. Indeed, assume to the contrary that $f$ is not an isometry. Then there exists  $y_0\in N$ such that $f_*\colon T_{y_0}N \to T\sph^{2k+1}$ is not an isometry.  Consider the singular value decomposition of $f_*\colon T_{y_0}N\to T_{f(y_0)}\sph^{2k+1}$, that is, there exist orthonormal bases $\{\overbar e_i\}_{1\leq i\leq 2k+1}$ of $T_{y_0}N$ and $\{e_i\}_{1\leq i\leq 2k+1}$ of $T_{f(y_0)}\sph^{2k+1}$ such that $f_*\overbar e_i=\mu_ie_i$ for some $\mu_i\geq 0$. Since $2k+1\geq  3$ and $\mu_i\mu_j \leq 1$ for all $i\neq j$,  there exist $1\leq \alpha, \beta\leq 2k+1$ with $\alpha\neq \beta$ such that   $$\mu_\alpha \mu_\beta < 1.$$  This together with line \eqref{eq:D^2-sphere2} implies that there is $\delta'>0$ such that
$$\frac{\Sc_{\overbar g}}{4}-\frac{1}{2}\sum_{k,l}\langle f_*\overbar w_k,w_l\rangle \, \overbar c(\overbar w_k)\otimes c(w_l) \geq \delta' - \frac{2k+1}{R}, \quad \quad \forall x\in\mathcal N_{\delta'}(y_0)\times \sph^1_R.$$
Now by applying the same estimates as in \eqref{eq:poincare3} and \eqref{eq:stokes}, we arrive at a contradiction. This finishes the proof.   
 \end{proof}

Similar remarks also apply to the following improvement of Llraull's theorem due to Listing \cite{Listing:2010te}.  
	\begin{theorem}[{Listing \cite{Listing:2010te}}]\label{thm:llarull-coeff}
		Let $(\mathbb S^{n}, g_{st})$ be the standard round unit sphere of odd  dimension $n\geq 3$. Let $(N,\overbar g)$ be a closed spin Riemannian manifold and  $f\colon N\to M$  a smooth map with non-zero degree. If  $\Sc_{\overbar g}\geq \|\Bigwedge^2f_*\|\cdot n(n-1)$,  then there exists a constant $a>0$ such that $f\colon (N,a\cdot \overbar g) \to (\mathbb S^{n}, g_{st})$ is an isometry. 
	\end{theorem}
\begin{proof}
	Let $S(TN\oplus f^*T\sph^n)$ be the spinor bundle of $TN\oplus f^*T\sph^n$ over $N$ and $D$ its Dirac operator. For each point $y_0\in N$ in $N$, we pick a singular value decomposition of $f_*\colon T_{y_0}N\to T_{f(y_0)}\sph^{n}$, that is, there exist orthonormal bases $\{\overbar e_i\}_{1\leq i\leq n}$ of $T_{y_0}N$ and $\{e_i\}_{1\leq i\leq n}$ of $T_{f(y_0)}\sph^{n}$ such that $f_*\overbar e_i=\mu_ie_i$ for some $\mu_i\geq 0$. By the Bochner--Lichnerowicz--Weitzenb\"ock formula, we have 
	$$D^2=\nabla^*\nabla+\frac{\Sc_{\overbar g}}{4}+\frac 1 8\sum_{i\ne j} \mu_i\mu_j\overbar c(\overbar e_i)\overbar c(\overbar e_j)\otimes c(e_i)c(e_j).$$
	We may assume that $\mu_1\leq\cdots\leq \mu_n$. Thus $\|\Bigwedge^2f_*\|=\mu_{n-1}\mu_n$, and
\begin{equation}\label{eq:sphere}
		\frac{\Sc_{\overbar g}}{4}+\frac 1 8\sum_{i\ne j} \mu_i\mu_j\overbar c(\overbar e_i)\overbar c(\overbar e_j)\otimes c(e_i)c(e_j)\geq \frac{\Sc_{\overbar g}}{4}-\frac{\mu_{n-1}\mu_n\cdot n(n-1)}{4} \geq 0.
\end{equation}
	The equality holds if and only if $\mu_i\mu_j=\mu_{n-1}\mu_n$ for all $i\ne j$. Since $n\geq 3$, $\mu_i\mu_j=\mu_{n-1}\mu_n$ for all $i\ne j$ implies that 
	\begin{enumerate}
		\item either all $\mu_i$'s are nonzero and equal to each other, 
		\item or $\mu_i =0$ for all $1\leq i\leq n-1$.  
	\end{enumerate} 
Since $n$ is odd, we consider the product with a large circle as in the proof of Theorem \ref{thm:llarull-odd}. If there exists a point $x\in N$ such that condition (1) and condition (2) above both fail, then line \eqref{eq:sphere} becomes a strict inequality in a small neighborhood of $x$ in $N$. Now the same argument from the proof of Theorem \ref{thm:llarull-odd} together with a Poincar\'e-type inequality leads to a contradiction. To summarize, for any point $x\in N$, either condition (1) or condition (2) holds at $x$. 
	
	The rest of the proof follows from the same argument in \cite[Theorem 1.6 (I)]{Wang:2021tq}. Let $U$ be the open set of $N$ where condition (1) holds. Since the degree of $f$ is non-zero, $U$ is non-empty. Let $h = \|\Bigwedge^2 f_\ast\|^{1/2}$. Then we have $h^2 \cdot  \overbar g = f^\ast g_{st}$ on $U$.  By the formula of scalar curvature under conformal change, we have on $U$
	$$f^*\Sc_{g_{st}}=\frac{\Sc_{\overbar g}}{h^2}-\frac{2(n-1)}{h^3}\Delta h-\frac{(n-1)(n-4)}{h^4}|dh|^2.$$
	Since $\Sc_{\overbar g}=h^2\cdot f^*\Sc_{g_{st}}$ and $h\equiv 0$ on $N-U$, we see that
	$$2h^k\Delta h=-(n-4)h^{k-1}|\nabla h|^2$$
	on the entire $N$ for all $k\geq 1$. By the Stokes theorem, we have
	$$0=\int_N\big(h^k\Delta h+\langle\nabla(h^k),\nabla h\rangle\big)=\int_N \big(h^k\Delta h+kh^{k-1}|\nabla h|^2\big).$$
	Therefore
	$$\Big(k-\frac{n-4}{2}\Big)\int_N h^{k-1}|\nabla h|^2=0$$
	for all $k\geq 1$. As a result, $\nabla h\equiv 0$ on $N$. Therefore, $h$ is a nonzero constant function, say $a$, on $N$. It follows that  $f\colon (N,a\cdot \overbar g) \to (\mathbb S^{n}, g_{st})$ is a local  isometry.  Since $n\geq 3$, the sphere $\sph^n$ is simply connected. Therefore,  $f\colon (N,a\cdot \overbar g) \to (\mathbb S^{n}, g_{st})$ is an isometry. This finishes the proof. 
\end{proof}

Now let us prove Theorem \ref{thm:extremalitySphere}. 

\begin{proof}[Proof of Theorem \ref{thm:extremalitySphere}]
	Since an even dimensional sphere has non-zero Euler characteristic, which has already been covered by Theorem \ref{thm:nonzeroEuler}, we shall only focus the case where $X$ is an odd dimensional standard round sphere. Our proof will be a combination of the proof of Theorem \ref{thm:llarull-odd} and the proof of Theorem \ref{thm:nonzeroEuler}.

	Similar to the proof of Proposition \ref{prop:special}, for a given $0<\lambda <c$, let 
 $M_\lambda = [-\lambda, \lambda]\times \sph^{n-1} \subset M$.  We assume without loss of generality that  $N_\lambda=f^{-1}(M_\lambda)$ is a manifold with boundary. We denote by $\widetilde M_\lambda=[-\lambda,\lambda]\times \sph^n$, equipped with the metric
	$$\widetilde g=dr^2+\varphi(r)^2\cdot g^{\sph^n}_{st}.$$

	Let $N_\lambda\times \sph^1_R$ be the Riemannian  product of $(N_\lambda, \overbar g)$ and the circle  $\mathbb S^1_R$ of radius $R$. 
	For any $\varepsilon',\varepsilon>0$, let $\expand$ be the smooth function given in line \eqref{eq:expand}: 
	\begin{equation}
		\expand\colon [-\gamma,\gamma]\to [-c,c]
	\end{equation}
	such that
	\begin{itemize}
		\item $\expand(\pm\gamma)=\pm c$,
		\item $1\leq \expand'(r)\leq 1+\varepsilon'$ for $r\in\mathcal N_\varepsilon(I_0)$, and
		\item $\expand'(r)=1$ for $r\in[-\gamma,\gamma]\backslash\mathcal N_\varepsilon(I_0)$.
	\end{itemize}
	where $I_0$ is a subinterval of $(-c, c)$ as given at the beginning of Subsection \ref{sec:estimates} and $\mathcal N_\varepsilon(I_0)$ is the $\varepsilon$-neighborhood of $I_0$.

	For $\lambda\in(0,\gamma)$ and $\mu=\expand(\lambda)$, we consider the map
	$$\widetilde h\colon (N_\lambda\times \sph^1,\overbar g+R^2d\theta^2)\to (\widetilde M_\mu,\widetilde g)$$
	defined as the composition of the following maps
	\begin{equation}\label{eq:h-tilde}
		\begin{split}
			&(N_\lambda\times \sph^1,\overbar g+R^2d\theta^2)\xlongrightarrow{h\times\id}(M_\mu\times \sph^1, dr^2+\varphi(r)^2g^{\sph^{n-1}}_{st}+R^2d\theta^2)\xlongrightarrow{\id} \\
			&(M_\mu\times \sph^1,dr^2+\varphi(r)^2(g^{\sph^{n-1}}_{st}+d\theta^2))\xrightarrow{\id_r\times\alpha}
			(\widetilde M_\mu,\widetilde g),
		\end{split}
	\end{equation}
where $h$ is given in the proof of Proposition \ref{prop:special} and $\alpha$ is the map from $\sph^{n-1}\times\sph^1$ to $\sph^n$ in the proof of Theorem \ref{thm:llarull-odd}.
	By construction, $\widetilde h$ is a smooth map with non-zero degree.
	
	Set $E=S(T(N_\lambda\times \sph^1)\oplus \widetilde h^*T\widetilde M_\mu)$. We impose the same  boundary condition $B$ on sections of $E$ at $\partial N_\lambda\times \sph^1$ as given in Definition \ref{def:boundaryCondition}. Let $\nabla$ be the spinorial connection on $E$ determined by the Levi-Civita connections of  $N_\lambda\times \sph^1$ and $\widetilde M_\mu$. A new connection $\tcon$ on $E$ is defined as follows
	\begin{equation}
		\tcon_\xi=\nabla_\xi+\frac 1 2 c(\nabla_{h_*\xi }^{\widetilde M}\partial_r)c(\partial_r),~\forall \xi\in C^\infty(N_\lambda\times \sph^1, T(N_\lambda\times \sph^1)).
	\end{equation}
	Let $\{\overbar e_i\}_{1\leq i\leq n+1}$ be a local orthonormal basis of $T(N_\lambda\times \sph^1)$, where $\{\overbar e_i\}_{1\leq i\leq n}$ is a local orthonormal basis of $TN_\lambda$ and $\overbar e_{n+1}$ is an orthonormal basis of $T\sph^1$. Let $\widehat D$ be the Dirac operator on $E$ with respect to $\tcon$, 
	\begin{equation}
		\tdirac=\sum_{i=1}^{n+1}\overbar c(\overbar e_i)\tcon_{\overbar e_i}.
	\end{equation}
	Similar to the proof of Proposition \ref{prop:special}, we write 
	\begin{equation}
		\Psi=\frac{n}{2}\psi(h^\ast r) \cdot\mathscr E c(\partial_r)
	\end{equation}
	where $\psi = \varphi'/\varphi$,  and define 
	\begin{equation}
		\widehat D_\Psi=\widehat D+\Psi.
	\end{equation}
We emphasize that here $n$ is equal to the dimension of $\widetilde M_\mu$ minus one.

	Since $\chi(\widetilde M_\mu) \neq 0$ and $\deg(\widetilde h)\neq 0$,  we have $\ind(\tdirac_\Psi)\ne 0$, cf. the proof  of Proposition \ref{prop:special}.  Therefore, there exists a non-zero section $\sigma$ of $E$ satisfying the boundary condition $B$ such that $\tdirac_\Psi\sigma=0$.
	
	Note that
	\begin{equation}\label{eq:Bsquare2}
		0=|\widehat D_\Psi\sigma|^2=|\widehat D\sigma|^2+\big(\langle\widehat D\sigma,\Psi\sigma\rangle+\langle\Psi\sigma,\widehat D\sigma\rangle\big)+|\Psi\sigma|^2.
	\end{equation}	
	By the Bochner--Lichnerowicz--Weitzenb\"ock formula, we have $\widehat D^2=\widehat\nabla^*\widehat\nabla+\mathcal R$.	Now we compute the curvature operator $\mathcal R$. 
	Let $\widetilde P$ be the orthogonal projection from $T\widetilde M$ onto  $(\partial_r)^\perp$, and $P$ the orthogonal projection from $TM$ onto $(\partial_r)^\perp$, where $(\partial_r)^\perp$ is the orthogonal complement of $\partial_r$ in $T\widetilde M$ (resp. $TM$).  By the Bochner--Lichnerowicz--Weitzenbock formula for $\tcon$, we have
	\begin{equation}\label{eq:R}
		\mathcal R=\frac{\Sc_{\overbar g}}{4}-\frac{1}{2\varphi(h^*r)^2}\sum_{i<j\leq n+1}\overbar c(\overbar e_i\wedge\overbar e_j) c(\widetilde P(\widetilde h_*\overbar e_i)\wedge \widetilde P(\widetilde h_*\overbar e_j)),
	\end{equation}
	where the Clifford action of $2$-forms is defined by
	$$c(u\wedge v)\coloneqq\frac 1 2 (c(u)c(v)-c(v)c(u)).$$
	
	We claim that
	\begin{equation}
		\sum_{i<j\leq n+1}\overbar c(\overbar e_i\wedge\overbar e_j) c(\widetilde P(\widetilde h_*\overbar e_i)\wedge \widetilde P(\widetilde h_*\overbar e_j))\leq \|\Bigwedge^2(\widetilde P\widetilde h_*)\|_1,
	\end{equation}
	where $\|\Bigwedge^2(\widetilde P\widetilde h_*)\|_1$ denotes the trace norm of the linear map
	$$\Bigwedge^2(\widetilde P\widetilde h_*)\colon \Bigwedge^2T(N_\lambda\times \sph^1)\to \Bigwedge^2T\widetilde M_\mu.$$
	
Indeed, let us consider the singular value decomposition of $\Bigwedge^2(\widetilde P\widetilde h_\ast)$. More precisely, there exist orthonormal bases $\{\overbar \omega_k\}$ of $\Bigwedge^2T(N_\lambda\times \sph^1)$, $\{\omega_k\}$ of $\Bigwedge^2T\widetilde M_\mu$, and $\lambda_k\geq 0$ for $1\leq k\leq n(n+1)/2$ such that
	$$(\widetilde P\widetilde h_*)\overbar\omega_k=\lambda_k\omega_k.$$
	Therefore, we have
	\begin{align*}
		&\sum_{i<j\leq n+1}\overbar c(\overbar e_i\wedge\overbar e_j) c(\widetilde P(\widetilde h_*\overbar e_i)\wedge \widetilde P(\widetilde h_*\overbar e_j))\\
		=&\sum_{k=1}^{n(n+1)/2}\overbar c(\overbar\omega_k) c(\omega_k)\lambda_k	\leq\sum_{k=1}^{n(n+1)/2}\lambda_k=\|\Swedge^2(\widetilde P\widetilde h_\ast)\|_1.
	\end{align*}
	This finishes the proof of the claim.
	
	Recall that $h\colon N_\lambda\to M_\mu$ is given by $h=h_\expand\circ f$, cf.  line \eqref{eq:hchi}. We define
	$$\beta\colon (M_\mu\times \sph^1,\overbar g+R^2d\theta^2)\to (\widetilde M_\mu,\widetilde g=g^{\sph^{n+1}}_{st})$$
	as the composition of the following maps
	\begin{equation*}
		\begin{split}
			(M_\mu\times \sph^1, g+R^2d\theta^2)\xlongrightarrow{\id} 
			(M_\mu\times \sph^1,dr^2+\varphi(r)^2(g^{\sph^{n-1}}_{st}+d\theta^2))\xrightarrow{\id_r\times\alpha}
			(\widetilde M_\mu,\widetilde g).
		\end{split}
	\end{equation*}
	Then we have
	$$\widetilde h=\beta\circ(h_\expand\times \id)\circ f.$$
	We notice that
	$$\widetilde P\widetilde h_*=(\widetilde P\beta_*P)\circ
	(P(h_{\expand}\times id)_*P)\circ f_*.$$
	For each $r\in(-c,c)$, we denote by $X_r$ the leaf of $M$ at $r$, and $\widetilde X_r$ the leaf of $\widetilde M$ at $r$. Then for each fixed $r$, we have
	$$\widetilde P\beta_* P=\beta_*\colon T(X_r\times \sph^1)\to T\widetilde X_r,$$
	and
	$$P(h_{\expand}\times \id)_*P=(h_{\expand}\times \id)_*\colon T(X_r\times \sph^1)\to T(X_{\expand(r)}\times \sph^1),$$
	as $h_\expand$ maps the leaf at $r$ to the leaf at $\expand(r)$.  Therefore, by the H\"{o}lder inequality, we have
	\begin{equation}
		\begin{split}
			\|\Swedge^2(\widetilde P\widetilde h_*)\|_1\leq& \|\Swedge^2(\widetilde P\beta_*P)\|_1\cdot \|\Swedge^2(P(h_\expand\times \id)_*P)\|\cdot\|\Swedge^2 f_*\|\\
			\leq &\frac{\varphi(h^*r)}{\varphi(f^*r)}\cdot\|\Swedge^2(\widetilde P\beta_*P)\|_1,
		\end{split}
	\end{equation}
	where we have used the fact $\|\Swedge^2 f_*\|\leq 1$ since $f$ is distance-non-increasing and $$\|\Bigwedge^2(P(h_\expand\times \id)_*P)\|=\frac{\varphi(h^*r)^2}{\varphi(f^*r)^2} $$
	since  $h_\expand$ maps the leaf at $r$ to the leaf at $\expand(r)$.

	Now we estimate the trace norm of
	$$\Bigwedge^2(\widetilde P\beta_*P)=\Bigwedge^2\beta_*\colon \Bigwedge^2T(X_r\times \sph^1)\to \Bigwedge^2T\widetilde X_r$$
	for each fixed $r$. Let $p$ be the orthogonal projection 
	$$p\colon \Bigwedge^2T(X_r\times \sph^1)\to TX_r\wedge T\sph^1=TX_r\otimes T\sph^1.$$
	Note that the range of $p$ has dimension $(n-1)$, and the orthogonal complement of the range of $p$ has dimension $(n-1)(n-2)/2$. We recall that $\beta$ is distance-non-increasing. Therefore
	\begin{equation}
		\|\Bigwedge^2\beta_*\|_1\leq \|(\Bigwedge^2\beta_*)\circ p\|_1+\|(\Bigwedge^2\beta_*)\circ (1-p)\|_1
		\leq \frac{(n-1)(n-2)}{2}+\frac{n-1}{R}
	\end{equation}
	
	To summarize, we have shown that
	\begin{equation}
		\begin{split}
			\mathcal R\geq& \frac{\Sc_{\overbar g}}{4}-\frac{1}{2\varphi(h^*r)^2}\cdot \frac{\varphi(h^*r)^2}{\varphi(f^*r)^2}\Big(\frac{(n-1)(n-2)}{2}+\frac{n-1}{R} \Big)\\
			=&\frac{\Sc_{\overbar g}}{4}-\frac{(n-2)(n-1)}{4\varphi(f^*r)^2}
			-\frac{n-1}{2R\varphi(f^*r)^2}.
		\end{split}
	\end{equation}
	
	The estimates for the other three terms in line \eqref{eq:Bsquare2} are the same as in the proof of Proposition \ref{prop:special}. Recall the operator $\mathcal Q$ defined on $N_\lambda$ as given in line \eqref{eq:Q}. Since $\tdirac_\Psi \sigma = 0 $, it follows that $\mathcal P\sigma  = \mathcal Q\sigma$.
	To summarize, we have
	\begin{equation}\label{eq:D^2-oddSph}
		\begin{split}
			0=&\int_{N_\lambda\times \mathbb S^1} |\tdirac_\Psi\sigma|^2\\
			\geq &\frac{n}{4(n-1)}\int_{N_\lambda\times \sph^1} |\mathcal Q\sigma|^2+\big(\Sc_{\overbar g}-f^*\Sc_{g}\big)|\sigma|^2	-\frac{n-1}{2R\varphi(f^*r)^2}| \sigma|^2\\
			&+\int_{\partial N_\lambda\times\sph^1}\big(\frac{n}{2(n-1)}H_{\overbar g}+ \frac{n}{2}|\psi(\mu)|\big)|\sigma|^2\\
			&-\frac{\varepsilon'n}{2}\int_{f^{-1}(\mathcal N_\varepsilon(I_0)\times \sph^{n-1})\times\sph^1}(|\psi'(h^*r)|+c_0)|\sigma|^2,
		\end{split}
	\end{equation}
where $c_0$ is independent of $\varepsilon,\varepsilon',\lambda,\mu$ and $R$, and $I_0$ is the interval defined at the beginning of Section \ref{sec:estimates}.
	We remark that we have used condition (3) of Definition \ref{def:admissible} for the warping function $\varphi$ in the above inequality. 
	
	Now we shall prove $\Sc_{\overbar g}=f^*\Sc_g$ by contradiction. Assume to the contrary that the inequality $\Sc_{\overbar g}\geq f^*\Sc_g$ is strict somewhere. More precisely, assume that there is $x_0\in N$ and $\delta>0$ such that
	$$\Sc_{\overbar g}(x)\geq f^*\Sc_{g}(x)+\delta,~\forall x\in\mathcal N_{\delta}(x_0).$$
	 Let $K$ be a compact domain in $N$ containing $\mathcal N_{\delta}(x_0)$ and $f^{-1}(\mathcal N_\varepsilon(I_0)\times \sph^{n-1})$, and $\widetilde K=K\times \sph^1$. For any $\varepsilon,\varepsilon'$ below, we always choose $\lambda$ such that
	\begin{equation}\label{eq:lambda-oddSph}
		\frac{n}{2(n-1)}H_{\overbar g}+ \frac{n}{2}|\psi(\mu)|\geq 1>0
	\end{equation}
	on $\partial N_\lambda$, where $\mu = \rho(\lambda)$. 
	
	Note that, as long as $\varepsilon$ and $\varepsilon'$  are sufficiently small,  there exists $c_1>0$ (independent of $\varepsilon,\varepsilon',\lambda$ and $\mu$) such that $|\psi'(h^*r)|+c_0\leq c_1$ on $f^{-1}(\mathcal N_\varepsilon(I_0))$, where $I_0$ is the interval as chosen in the beginning of Section \ref{sec:estimates}.  Thus line \eqref{eq:D^2-oddSph} yields that
	\begin{equation}\label{eq:D^2-oddSph2}
		\begin{split}
			0\geq&\frac{n}{4(n-1)}\int_{N_\lambda\times \sph^1} |\mathcal Q\sigma|^2	-\frac{n-1}{2R\varphi(f^*r)^2}| \sigma|^2\\
			&+\frac{n\delta}{4(n-1)}\int_{\mathcal N_\delta(x_0)\times\sph^1}|\sigma|^2-\frac{\varepsilon'nc_1}{2}\int_{\widetilde K}|\sigma|^2.
		\end{split}
	\end{equation}
It follows that
\begin{equation*}
	\begin{cases}
		\displaystyle \int_{N_\lambda\times\sph^1}|\mathcal Q\sigma|^2\leq \frac{2(n-1)^2}{nR}\int_{N_\lambda\times\sph^1}\frac{1}{\varphi(f^*r)^2}|\sigma|^2+2c_1\varepsilon'(n-1)\int_{\widetilde K}|\sigma|^2,\vspace{.5cm}\\
		\displaystyle \delta\int_{\mathcal N_\delta(x_0)\times\sph^1}|\sigma|^2\leq \frac{2(n-1)^2}{nR}\int_{N_\lambda\times\sph^1}\frac{1}{\varphi(f^*r)^2}|\sigma|^2+2c_1\varepsilon'(n-1)\int_{\widetilde K}|\sigma|^2.
	\end{cases}
\end{equation*}
Therefore, we have 
\begin{equation}\label{eq:3.22}
\begin{split}
		&\frac 2 3 \int_{N_\lambda\times\sph^1}|\mathcal Q\sigma|^2+\frac 1 3\int_{\widetilde K}|\mathcal Q\sigma|^2+\int_{\mathcal N_\delta(x_0)\times\sph^1}|\sigma|^2\\
		\leq~& \varepsilon'\cdot 2c_1(n-1)\Big(1+\frac 1 \delta\Big)\int_{\widetilde K}|\sigma|^2\\
		&+\frac 1 R\cdot \frac{2(n-1)^2}{n}\Big(1+\frac 1 \delta\Big)\cdot\sup_{N_\lambda}\frac{1}{\varphi(f^*r)^2}\cdot \int_{N_\lambda\times\sph^1}|\sigma|^2
\end{split}
\end{equation}

By Lemma \ref{lemma:poincare}, we have the Poincar\'e-type inequalities
\begin{equation}\label{eq:poincare2}
	\begin{cases}
		\displaystyle 	\int_{\widetilde K}|\sigma|^2\leq C\int_{\mathcal N_\delta(x_0)\times\sph^1}|\sigma|^2+C\int_{\widetilde K}|\mathcal Q\sigma|^2, \vspace{.5cm}\\
		\displaystyle \int_{N_\lambda\times\sph^1}|\sigma|^2\leq C'\int_{\mathcal N_\delta(x_0)\times\sph^1}|\sigma|^2+C'\int_{N_\lambda\times\sph^1}|\mathcal Q\sigma|^2.
	\end{cases}
\end{equation}
  The readers should not confuse the two constants $C$ and $C'$ in the two inequalities of \eqref{eq:poincare2} above. The first inequality of \eqref{eq:poincare2} only requires the geometric data of $f\colon N\to M$ near $K$, which is not affected by the $\sph^1$-direction. Hence the constant $C>0$ only depends on $K$ and $\delta$, and is independent of $\varepsilon,\varepsilon',\lambda,\mu$ and $R$. The second inequality of \eqref{eq:poincare2} requires the geometric data of the entire $N_\lambda$. In particular, the constant $C'>0$ may depend on $\varepsilon,\varepsilon',\lambda,\mu$, but is still independent of $R$.

	Now given the compact domain $K$ and $\delta>0$, we choose  $\varepsilon>0$ and $\varepsilon' >0$ small enough so that  
	$$2Cc_1\varepsilon'(n-1)\Big(1+\frac{1}{\delta}\Big)\leq \frac 1 3,$$
	and choose $\lambda$ as in line \eqref{eq:lambda-oddSph}. With $\varepsilon,\varepsilon',\lambda$ chosen, the constant $C'$ is now fixed. Finally, since $N_\lambda$ is also fixed and compact, we choose $R$ large enough so that 
	$$\frac 1 R\cdot \frac{2C'(n-1)^2}{n}\Big(1+\frac 1 \delta\Big)\cdot \sup_{N_\lambda}\frac{1}{\varphi(f^*r)}\leq \frac 1 3.$$
	It follows from \eqref{eq:3.22} and \eqref{eq:poincare2} that
	$$\frac 1 3 \int_{N_\lambda\times\sph^1}|\mathcal Q\sigma|^2+\frac 1 3\int_{\mathcal N_\delta(x_0)\times\sph^1}|\sigma|^2\leq 0.$$
	This together with the second inequality in line \eqref{eq:poincare2} implies that  $\sigma$ vanishes on $N_\lambda\times\sph^1$, which leads to a contradiction. Therefore, we have proved that $\Sc_{\overbar g}=f^*\Sc_g$. 
	
	Now if in addition $n\geq 3$ and $\varphi$ is strictly log-concave, then the proof of  the \textbf{Scalar rigidity} part of Theorem \ref{thm:nonzeroEuler} or the proof of the rigidity part of Theorem \ref{thm:llarull-odd} can be easily  adapted to the current setting to show that $f$ is an isometry. We shall not repeat the details. This completes the proof.  
\end{proof}

\section{Scalar curvature rigidity of degenerate toric bands}\label{sec:torus}
In this section, we first prove Theorem \ref{thm:extremalityTorus-intro}, which is an improvement of Theorem  \ref{thm:extremality} for the case where the leaf $X$ of $M$ is a flat torus. The general case of  Theorem \ref{thm:extremality} will then follow by a combination of the proofs of Theorem \ref{thm:nonzeroEuler}, Theorem \ref{thm:extremalitySphere} and Theorem \ref{thm:extremalityTorus-intro}. 
\begin{proof}[Proof of Theorem \ref{thm:extremalityTorus-intro}]
	Without loss of generality, we assume that $(n-1)$ is even. The case where $(n-1)$ is odd can be proved similarly by taking product with a circle as in the proof of Theorem \ref{thm:extremalitySphere}.
	
	
	The torus $\mathbb T^{n-1}$ is enlargable \cite{MR569070}. More precisely, for any $\epsilon>0$, there exists a finite-sheeted covering space of $\mathbb T^{n-1}$ (equipped with the lifted metric) which admits an $\epsilon$-contracting map onto the standard round sphere $\sph^{n-1}$  such that the map is constant at infinity and of non-zero degree. In particular, for a   finite covering space  $\mathbb T_\Lambda^{n-1}$  of $\mathbb T^{n-1}$, let us denote this  $\epsilon$-contracting map by  $\vartheta_\Lambda\colon \mathbb T^{n-1}_\Lambda\to \sph^{n-1}$. Let $(-c, c)\times \sph^{n-1}$ be the Riemannian product of $\sph^{n-1}$ with the interval $(-c, c)$. Consider the map 
	\[ \id \times \vartheta_\Lambda \colon M  = (-c,c) \times \mathbb T^{n-1}_\Lambda \to (-c, c)\times \sph^{n-1}. \] 
	Note that the metric of the leaf $\{r\}\times \mathbb T^{n-1}_\Lambda$ has be rescaled by a factor of $\varphi(r)^2$. But in any case,  for any $\epsilon>0$ and any $ 0<\ell <c$,  there exists a sufficiently large finite-sheeted covering space $\mathbb T^{n-1}_{\Lambda}$ of $\mathbb T^{n-1}$ such that 
	\[ \Theta_\Lambda\coloneqq \id \times \vartheta_\Lambda \colon  [-\ell, \ell] \times \mathbb T^{n-1}_\Lambda \to [-\ell, \ell]\times \sph^{n-1}. \] 
	is $\epsilon$-contracting and of non-zero degree. 
	
	Let $N_\Lambda$ be the covering space over $N$ induced by the covering space 
	\[ (-c, c)\times \mathbb T^{n-1}_\Lambda \to M = (-c, c)\times \mathbb T^{n-1}\] via the map $f\colon N \to M$. The map $f$ lifts to a map $N_{\Lambda}\to (-c,c)\times \mathbb T^{n-1}_\Lambda$, which we still denote by $f$.	Let $S_{N_\Lambda}$ be the spinor bundle of $(N_\Lambda,\overbar g)$.  
	Set $h\coloneqq h_\expand\circ f$ as in line \eqref{eq:hchi}, where the function $\expand$ is defined similarly as in line \eqref{eq:expand}. More precisely, for any $\varepsilon',\varepsilon>0$, there is $0<\gamma<c$ and a smooth function
	\begin{equation}
		\expand\colon [-\gamma,\gamma]\to [-c,c]
	\end{equation}
	such that
	\begin{itemize}
		\item $\expand(\pm \gamma)=\pm c$,
		\item $1\leq \expand'(r)\leq 1+\varepsilon'$ if $r\in\mathcal N_\varepsilon(I_0)$, and
		\item $\expand'(r)=1$ for $r\in [-\gamma,\gamma]\backslash\mathcal N_\varepsilon(I_0)$.
	\end{itemize}	
	where $I_0$ is a subinterval of $(-c, c)$ chosen as at the beginning of Subsection \ref{sec:estimates}. 
	
	Let $\lambda >0$ be sufficiently close to $\gamma$ and $\mu=\expand(\lambda)$.  We denote by 
	\[ N_{\Lambda,\lambda}=f^{-1}([-\lambda,\lambda]\times \mathbb T^{n-1}_\Lambda)\]  and $M_{\Lambda,\mu}=[-\mu,\mu]\times \mathbb T^{n-1}_\Lambda$. By a similar discussion as in Subsection \ref{sec:estimates},  without loss of generality, we may assume that the preimage of $X_r$ is a smooth submanifold of $N$ for $r$ close enough to $\pm c$. Therefore, without loss of generality, we may assume $N_{\Lambda,\lambda}$ is a smooth manifold with boundary. 
	
		Now let us set $E$ to be the spinor bundle 
	$$E=S\big(T{N_{\Lambda,\lambda}}\oplus (\Theta_\Lambda\circ h)^\ast T([-\mu, \mu]\times \sph^{n-1}) \big)   $$
	over $N_{\Lambda,\lambda}$, where $T([-\mu, \mu]\times \sph^{n-1})$ is the tangent bundle of $[-\mu, \mu]\times \sph^{n-1}$. 
	By construction, the pull-back bundle $(\Theta_\Lambda\circ h)^\ast T([-\mu, \mu]\times \sph^{n-1})$ gets arbitrarily flat over $N_{\Lambda,\lambda}$ as $\Lambda$ becomes sufficiently large.  That is, we may assume the curvature of $(\Theta_\Lambda\circ h)^\ast T([-\mu, \mu]\times \sph^{n-1})$ to be as small as we wish, as long as $\Lambda$ is sufficiently large. 
	
	Similar as the proof of Proposition \ref{prop:special}, we consider a specific  Dirac operator together with potential on $N_{\Lambda, \lambda}$ as follows. Let $\partial_r$ be the unit vector in $(\Theta_\Lambda\circ h)^\ast T([-\mu, \mu]\times \sph^{n-1})$ along the direction of $[-\mu, \mu]$. Let $\nabla$ be the spinorial connection on $E$ naturally induced by  the Levi--Civita connection on $N$ and the pull-back of the Levi--Civita connection on $M$. Similar to line \eqref{eq:tcon},  we introduce a new connection on $E$  by
	\begin{equation*}
		\widehat\nabla_\xi\coloneqq \nabla_\xi+\frac 1 2 c(\nabla^{\sph^{n-1}}_{(\Theta_\Lambda\circ h)_*\xi}\, \partial_r)c(\partial_r),
	\end{equation*}
	where $\nabla^{\sph^{n-1}}$ is the Levi--Civita connection of $[-\mu, \mu]\times \sph^{n-1}$. Note that, since $[-\mu, \mu]\times \sph^{n-1}$ is a Riemannian product,  $\partial_r$ is clearly  parallel with respect to the connection $\nabla^{\sph^{n-1}}$. In other words, the ``new'' connection $\widehat\nabla$ above in fact is equal to the original connection $\nabla$ on $E$. We only introduced the new connection so that our notation is more consistent with that from Subsection \ref{sec:estimates}.    Let $\widehat D$ be the Dirac operator on $E$ with respect to $\widehat\nabla$,
	$$\widehat D=\sum_{i=1}^n \overbar c(\overbar e_i)\widehat \nabla_{\overbar e_i}$$
	where $\{\overbar e_i\}_{1\leq i\leq n}$ is  local orthonormal basis of $TN_{\Lambda, \lambda}$.
	
	Recall that we have 
	\begin{equation*}
		\psi=\frac{\varphi'}{\varphi}=(\log\varphi)'.
	\end{equation*}
	We denote by $r\colon M = (-c, c) \times X \to (-c, c)$ the projection to the first component, that is, $r$ maps the leaf $X_t$ to $t$.
	We set 
	\begin{equation*}
		\Psi\coloneqq \frac{n}{2}\cdot  \psi(h^\ast r) \cdot \mathscr E \cdot c(\partial_r), 
	\end{equation*}
	where $ \mathscr E$ is the $\mathbb Z_2$-grading on $E$ and $h^\ast r$ is the function $r\circ h \colon N_\lambda \to [-\mu, \mu]$.  We define 
	\begin{equation}
		\widehat D_\Psi\coloneqq \widehat D+\Psi.
	\end{equation}
	and impose the same boundary condition $B$ as in Definition \ref{def:boundaryCondition}.
 
 Since both $\deg(\Theta_\Lambda\circ h)\neq 0$ and $\chi(\sph^{n-1}) = 2\neq 0$, we have 
 $$\ind(\widehat D_\Psi)=\deg(\Theta_\Lambda\circ h)\cdot\chi([-\mu, \mu]\times \sph^{n-1}) \ne 0.$$
 There exists a non-zero section $\sigma$ of $E$ over $N_{\Lambda, \lambda}$ satisfying the boundary condition $B$ such that  $\tdirac_\Psi\sigma=0$.

Now let us prove $\Sc_{\overbar g} =  f^*\Sc_{g}$ by contradiction. Indeed, we shall show that $\tdirac_\Psi$ is invertible if the inequality $\Sc_{\overbar g} \geq f^*\Sc_{g}$ is strict  somewhere on $N$. By construction,  the curvature of $(\Theta_\Lambda\circ h)^\ast T([-\mu, \mu]\times \sph^{n-1})$ is arbitrarily  small, as long as $\Lambda$ is sufficiently large.   Therefore, if $\sigma$ is a non-zero section of $E$ satisfying the boundary condition $B$ and $\tdirac_\Psi\sigma=0$, then the same proof of  Proposition \ref{prop:D^2} shows that 
	\begin{equation}\label{eq:D^2all-torus}
		\begin{split}
			0=\int_{N_{\Lambda,\lambda}} |\tdirac_\Psi\sigma|^2&\geq \frac{n}{4(n-1)}\int_{N_{\Lambda,\lambda}} |\mathcal Q\sigma|^2+(\Sc_{\overbar g}-f^*\Sc_{g})|\sigma|^2-C_\Lambda|\sigma|^2\\
			&+\int_{\partial N_{\Lambda,\lambda}}\big(\frac{n}{2(n-1)}H_{\overbar g}+ \frac{n}{2}|\psi(\mu)|\big)|\sigma|^2\\
			&-\frac{\varepsilon'n}{2}\int_{f^{-1}(\mathcal N_\varepsilon(I_0)\times\mathbb T^{n-1})}(|\psi'(h^*r)|+c_0)|\sigma|^2,
		\end{split}
	\end{equation}
	where $c_0$ is the same constant as line \eqref{eq:h>=f-c0} and $C_\Lambda$ is a positive constant such that $C_\Lambda\to 0$ as $\Lambda\to\infty$. We emphasize that we have used condition (3) of Definition \ref{def:admissible} for the warping function $\varphi$ in the above inequality \eqref{eq:D^2all-torus}. On the other hand, in the current case, the proof of the above inequality \eqref{eq:D^2all-torus} does \emph{not} require Lemma \ref{lemma:curvature>=}. This is because the scalar curvature of $M$ is calculated  only using the potential $\Psi$ in the current case, and the curvature term coming from $(\Theta_\Lambda\circ h)^\ast T([-\mu, \mu]\times \sph^{n-1})$ is arbitrarily small and has been reflected in the constant $C_\Lambda$ during the estimates. 
	
	Suppose that there is a point $x_0\in N$ such that $\Sc_{\overbar g}>f^*\Sc_g+\delta$ on $\mathcal N_\delta(x_0)$. Let $\Lambda x_0 \subset N_\Lambda$ be the preimage of $x_0$ via the covering map  $N_\Lambda \to N$. Then we also have $\Sc_{\overbar g}>f^*\Sc_g+\delta$ on $\mathcal N_\delta(\Lambda x_0)$. 
	
	Fix a compact set $K$ in $N$ that contains $\mathcal N_\delta(x_0)$ and $f^{-1}(\mathcal N_\varepsilon(I_0)\times X)$, and we denote by $K_\Lambda$ the preimage of $K$ in $N_\Lambda$. Given $\varepsilon>0$ and $\varepsilon'>0$, we choose $\lambda >0$ such that 
	\begin{equation}\label{eq:boundarypositive}
\frac{n}{2(n-1)}H_{\overbar g}+ \frac{n}{2}|\psi(\mu)|\geq 1>0,
	\end{equation}
	where $\mu = \rho(\lambda)$. 
	Without loss of generality,  we may assume that $|\psi'|\leq 1$ on $\mathcal N_\varepsilon(I_0)$. It follows from  the inequality  \eqref{eq:D^2all-torus} that 
	\begin{equation}\label{eq:comparision-torus}
		\begin{cases}
			\displaystyle\int_{N_{\Lambda,\lambda}}|\mathcal Q\sigma|^2\leq \frac{(1+c_0)\varepsilon'(n-1)}{2}\int_{K_\Lambda}|\sigma|^2+C_\Lambda\int_{N_{\Lambda,\lambda}}|\sigma|^2, \vspace{0.2cm}\\
			\displaystyle\delta\int_{\mathcal N_\delta(\Lambda x_0)}|\sigma|^2\leq 2(1+c_0)\varepsilon'(n-1)\int_{K_\Lambda}|\sigma|^2+C_\Lambda\int_{N_{\Lambda,\lambda}}|\sigma|^2.
		\end{cases}
	\end{equation}
	By Lemma \ref{lemma:poincare}, we have the following inequalities 
	\begin{equation}\label{eq:poincare-torus}
		\begin{cases}
			\displaystyle\int_{K_\Lambda}|\sigma|^2\leq
			C\int_{\mathcal N_\delta(\Lambda x_0)}|\sigma|^2+C\int_{K_\Lambda}|\mathcal Q\sigma|^2, \vspace{0.2cm} \\
			\displaystyle\int_{N_{\Lambda,\lambda}}|\sigma|^2\leq
			C'\int_{K_\Lambda}|\sigma|^2+C'\int_{N_{\Lambda,\lambda}}|\mathcal Q\sigma|^2.
		\end{cases}
	\end{equation}
	for some $C>0$ and $C'>0$. Note that there is $d>0$ such that the $d$-neighborhood of $\Lambda x_0$ in $N_\Lambda$ covers $K_\Lambda$, where $d$ is independent of $\Lambda$. Moreover, the geometric data of $E$ over $N_{\Lambda}$ restricted on $K_\Lambda$ are also independent of $\Lambda$. Therefore, the constant $C>0$ in the first line of \eqref{eq:poincare-torus} is independent of $\varepsilon,\varepsilon',\lambda$ and especially independent of $\Lambda$. The constant $C'$ in the second line of \eqref{eq:poincare-torus} is also independent of $\Lambda$, but  may depend on $\varepsilon,\varepsilon',\lambda$, as one of the integrals takes place on the entire $N_{\Lambda,\lambda}$. 
	
	Now we first choose $\varepsilon$ and $\varepsilon'$ to be sufficiently small  according to $C$, then choose $\lambda>0$ such that the inequality \eqref{eq:boundarypositive} holds, and finally choose $\Lambda$ to be sufficiently large according to $C'$. It is not difficult to see that an appropriate choice of $\varepsilon,\varepsilon',\lambda$ and $\Lambda$ leads to a contradiction. See for example the argument towards the end of the proof of Theorem \ref{thm:extremalitySphere}. This finishes the proof of the equality  $\Sc_{\overbar g} =  f^*\Sc_{g}$. 
	
	If in addition  $\varphi$ is strictly log-concave, then the proof of  the \textbf{Scalar rigidity} part of Theorem \ref{thm:nonzeroEuler} can be easily  adapted to the current setting to show that $N=(-c,c)\times Y$, the map $f$ respects  the product structures, and the metric  $\overbar g$ is also a warped product metric of the form \[ \overbar g = dr^2+\varphi(r)^2g_Y,  \]
	where $g_Y$ is a  metric on $Y$. It remains to show that $g_Y$ is flat. 
	
	As we have showed that $\Sc_{\overbar g} =  f^*\Sc_{g}$, the standard formula for the scalar curvature of warped product metrics (cf. line \eqref{eq:scalar-warp}) shows that $g_Y$ is scalar flat, that is,  $\Sc_{g_Y} \equiv 0$. Note that $Y$ maps to $\mathbb T^{n-1}$ with nonzero degree. Therefore $Y$ is enlargable \cite{MR569070}. It then follows from a theorem of Gromov and Lawson that   $Y$ does not admit a metric of positive scalar curvature,  and any metric of nonnegative scalar curvature on $Y$ is flat \cite[Theorem A]{MR569070}\cite[Chapter IV, Proposition 5.8]{spingeometry}. This finishes the proof. 
\end{proof}

Now the general case of Theorem \ref{thm:extremality} follows from a combination of the proofs of Theorem \ref{thm:nonzeroEuler}, Theorem \ref{thm:extremalitySphere} and Theorem \ref{thm:extremalityTorus-intro}.  

As we have seen in various steps of the proof of Theorem \ref{thm:extremality}, the notion of  admissible warping fucntions, introduced in  Definition \ref{def:admissible},  is crucial for the validity of scalar curvature extremality and rigidity of degenerate warped product spaces.  The log-concavity of $\varphi$ is a commonly expected necessary condition for the scalar curvature extremality and rigidity of warped product spaces. However,  condition (3) of Definition \ref{def:admissible} is new and  has not been previously considered in the literature regarding scalar curvature extremality and rigidity. The following Example \ref{ex:nonrigid} shows that condition (3) in fact is necessary.  More precisely, Example \ref{ex:nonrigid} shows that if we drop condition (3), then scalar curvature extremality and rigidity \emph{fail} for certain degenerate toric bands with warping functions satisfying conditions (1) and (2). 

\begin{example}\label{ex:nonrigid}
	Let $b$ be a positive number. Let $X$ be a flat torus $\mathbb T^{n-1}$ and $M=(-\pi/2,\pi/2)\times \mathbb T^{n-1}$, which carries a warped product metric 
	$$g=dr^2+\cos^{2b}(r)g_{\mathbb T^{n-1}}$$
	with the warping function $\varphi(r)=\cos^{b}(r)$. A direct computation shows that
	$$\psi\coloneqq (\log\varphi)'=-b\tan(r),  $$ $$\psi'=(\log\varphi)''=-b\sec^2(r),$$
	and
	$$\Sc_{g}=-2(n-1)(\psi'+\frac n 2\psi^2)=(n-1)\big(b(2-nb)\tan^2r+2b\big).$$
	In particular, $\varphi$ is strictly log-concave; and $\varphi(r)>0$ for $r\in(-\pi/2, \pi/2)$ such that  $\lim_{r\to \pm \pi} \varphi(r)=0$. Therefore $\varphi$ satisfies conditions (1) and (2) of Definition \ref{def:admissible}.  In the present case,  $\varphi$ satisfies condition (3) of Definition \ref{def:admissible}:
	$$\begin{cases}
		(\psi'+n\psi^2/2)'\leq 0&\text{ near }r=-\pi/2\\
		(\psi'+n\psi^2/2)'\geq 0&\text{ near }r=\pi/2
	\end{cases}$$
	if and only if $b\geq 2/n.$
	
	Now assume that $b<1/n$ and choose $\overbar b\in(b,1/n)$. Consider the following warped product metric 
	$$\overbar g=dr^2+\cos^{2\overbar b}(r)g_{\mathbb T^{n-1}}$$
	 on $M$. Note that the function $r$ is $1$-Lipschitz with respect to both $\overbar g$ and $g$. Hence all assumptions, except condition (3) of  Definition \ref{def:admissible}, of Theorem \ref{thm:extremalityTorus-intro} are satisfied by $(N, \overbar g) \coloneqq (M, \overbar g)$ and $(M, g)$.  However, a direct computation shows that 
	$$\Sc_{\overbar g}-\Sc_g=(n-1)\big(\overbar b(2-n\overbar b)-b(2-nb)\big)\tan^2r+2(n-1)(\overbar b-b)>0.$$
	Therefore, the above choice of $b$ and $\overbar b$  gives a counter-example of Theorem \ref{thm:extremalityTorus-intro}, if we drop  condition (3) in Definition \ref{def:admissible}.
\end{example}

We observe that the warping function $\varphi(r) = \cos^{b}(r)$ in Example \ref{ex:nonrigid} above satisfies conditions (1) and (2) of Definition \ref{def:admissible} for all $b>0$. It satisfies condition (3) of Definition \ref{def:admissible} if and only if $b\geq 2/n$. However, in Example \ref{ex:nonrigid}, we have chosen $b<1/n$ in order to demonstrate  the necessity of condition (3). This raises a natural question. 

\begin{question}
	Does scalar curvature rigidity hold for $M=(-\pi/2,\pi/2)\times \mathbb T^{n-1}$ equipped with the warped product metric 
	$$g=dr^2+\cos^{2b}(r)g_{\mathbb T^{n-1}}$$
	when $ 1/n \leq b <2/n$? 
\end{question}

\section{Scalar-mean rigidity of warped product spaces}\label{sec:scalar-mean}
In this section, we prove Theorem \ref{thm:scalar-mean-intro}. The proof is a straightforward adaption of the proof of Theorem \ref{thm:extremality}. 
\begin{proof}[Proof of Theorem \ref{thm:scalar-mean-intro}]
	For simplicity, we shall only focus on the case where the leaf $X$ has non-zero Euler characteristic.  The general case can be dealt with similarly as the general case of Theorem \ref{thm:extremality}.

As the main ingredients of the proof are very similar to those used in the proof Theorem \ref{thm:extremality}, we shall be brief. We start from the function $\expand$ as given in line \eqref{eq:expand} except that 
$\expand$ is defined to equal the identity map near $-c$ this time. 
	
	We retain the same notation from the proofs of Proposition \ref{prop:special} and  Theorem \ref{thm:nonzeroEuler}. Note that $\partial N_\lambda$ consists of two parts, where $\partial_-N_\lambda=\partial N$ is mapped to $\{-c\}\times X$, and the remaining part $\partial_+N_\lambda$ is mapped to $\{\mu\}\times X$. Let $\sigma$ be a non-zero section of the spinor bundle $E$ satisfying the boundary condition $B$ and $\tdirac_\Psi\sigma=0$. Then by Proposition \ref{prop:D^2}, we have
	\begin{equation}\label{eq:D^2all-mean}
		\begin{split}
			0=\|\tdirac_\Psi\sigma\|^2\geq&\frac{n}{n-1}\int_{N_\lambda} |\mathcal P\sigma|^2+\frac{n}{4(n-1)}\int_{N_\lambda}(\Sc_{\overbar g}-f^*\Sc_{g})|\sigma|^2\\
			&-\frac{\varepsilon' n}{2}\int_{f^{-1}(\mathcal N_\varepsilon(I_0)\times X)}(|\psi'(h^*r)|+c_0)|\sigma|^2\\
			&+\frac{n}{2(n-1)}\int_{\partial_- N_\lambda}\big(H_{\overbar g}-f^*H_g\big)|\sigma|^2\\
			&+\int_{\partial_+ N_\lambda}\big(\frac{n}{2(n-1)}H_{\overbar g}+
			\frac{n}{2}|\psi(h^*r)|\big)|\sigma|^2
		\end{split}.
	\end{equation}
	
	The equality of scalar curvature  $\Sc_{\overbar g}=f^*\Sc_g$ follows from the same argument of Theorem \ref{thm:nonzeroEuler}. Indeed,  otherwise the inequality of scalar curvature  $\Sc_{\overbar g}\geq f^*\Sc_g$ is strict somewhere, then  line \eqref{eq:D^2all-mean} and Lemma \ref{lemma:poincare} lead to a contradiction.

	Now let us prove  $H_{\overbar g}= f^*H_g$. Suppose to the contrary that  the inequality $H_{\overbar g}\geq  f^*H_g$ is strict somewhere on $\partial_-N_\lambda=\partial N$. That is,  there is a small open subset $\mathcal N$ in  $\partial N$ such that  $H_{\overbar g}\geq f^*H_g+\delta$ on $\mathcal N$ for some $\delta>0$. Let $K$ be a compact domain in $N_\lambda$ containing both $\partial N$ and  $f^{-1}(\mathcal N_\varepsilon(I_0)\times X)$. Then an obvious modification of the proof of Lemma \ref{lemma:poincare} shows that 
	\begin{equation}\label{eq:Poincare-codim1}
		\int_{K}|\sigma|^2\leq C\int_{\mathcal N}|\sigma|^2+C\int_{K}|\nabla\sigma|^2
	\end{equation}
	for some $C>0$ independent of the parameters $\varepsilon, \varepsilon', \lambda, \mu$ that appear in the construction of the function $h\colon N_\lambda\to M_\mu$. Indeed, the inequality \eqref{eq:Poincare-codim1} follows from the same proof of Lemma \ref{lemma:poincareEu} and Lemma \ref{lemma:poincare}, except that we replace line \eqref{eq:Rn} by 
	\begin{equation}\label{eq:Rn-codim1}
			\int_{I^{n-1}\times\{t\}}|\alpha|^2\leq e^{(2M+1)\ell}\Big(\int_{I^{n-1}\times\{0\}}|\alpha|^2+\int_K|\beta|^2\Big).
	\end{equation}
	for smooth function $\alpha$ over $\R^n$ and $\beta=\frac{d\alpha}{dx_n}+A\alpha$, and integrate with respect to $t\in[0,\ell]$ as in the  proof of  Lemma \ref{lemma:poincareEu}. Now the inequality \eqref{eq:Poincare-codim1}, together with  line \eqref{eq:D^2all-mean}, shows that $\sigma$ vanishes on $N_\lambda$, which contradicts the fact that $\sigma$ is nonzero.  This proves that $H_{\overbar g}= f^*H_g$. 
	
	The scalar rigidity part of the theorem follows by the same argument as the \textbf{Scalar rigidity} part of Theorem \ref{thm:nonzeroEuler}. This completes the proof. 
\end{proof}

\end{document}